\documentclass[reqno]{amsart}
\usepackage{graphicx}
\usepackage{mathrsfs}
\usepackage{bm}
\usepackage{bbm}
\usepackage{amsthm}
\usepackage[colorlinks,linkcolor=blue,anchorcolor=red,citecolor=blue]{hyperref}
\usepackage{geometry}
\newtheorem{Thm}{Theorem}[section]
\newtheorem{Def}[Thm]{Definition}
\newtheorem{Lem}[Thm]{Lemma}

\numberwithin{equation}{section}

\newcommand{\1}{\mathbf{1}}
\newcommand{\R}{\mathbb{R}}

\newcommand{\Z}{\mathbb{Z}}
\renewcommand{\P}{\mathbf{P}}
\renewcommand{\Re}{\text{Re}}
\renewcommand{\Im}{\text{Im}}
\newcommand{\F}{\mathcal{F}}

\newcommand{\C}{\mathbb{C}}
\newcommand{\E}{\mathcal{E}}

\renewcommand{\L}{\mathcal{L}}
\renewcommand{\H}{\mathcal{H}}
\newcommand{\g}{\mathbf{g}}
\renewcommand{\b}{\mathbf{b}}

\newcommand{\ve}{\varepsilon}

\newcommand{\wt}{\widetilde}
\newcommand{\ol}{\overline}
\newcommand{\dis}{\displaystyle}
\newcommand{\pa}{\partial}
\newcommand{\na}{\nabla}
\newcommand{\de}{\delta}

\newcommand{\lam}{\lambda}
\newcommand{\wh}{\widehat}

\usepackage{cite}

\allowdisplaybreaks
\renewcommand{\Re}{\mathrm{Re}}
\renewcommand{\S}{\mathbb{S}}
\newcommand{\<}{\langle}
\renewcommand{\>}{\rangle}
\newcommand{\T}{\mathbb{T}}
\newcommand{\I}{\mathbf{I}}

\renewcommand{\P}{\mathbf{P}}

\usepackage{tikz,pgfplots}
\usepackage{tikz-cd}
\usepackage{tikz-3dplot}
\usepackage{amsmath,amssymb,amsthm,amsfonts,bm}
\usepackage{color}
\usepackage{tcolorbox}
\usepackage{pdfpages}
\usepackage{textcomp}
\usetikzlibrary{patterns}
\usetikzlibrary{shapes.geometric}
\usetikzlibrary{arrows.meta,arrows}
\usetikzlibrary{decorations.pathreplacing,decorations.pathmorphing}
\usetikzlibrary{hobby}
\usetikzlibrary{math}
\usepgfplotslibrary{fillbetween}
\usepackage{graphicx}
\usepackage{multicol}
\usepackage{float}
\usepackage{subfigure}

\begin{document}
	
	\title[Spectral gap formation with soft potentials]{Spectral Gap Formation to Kinetic Equations with Soft Potentials in Bounded Domain}
	
	\author[D.-Q. Deng]{Dingqun Deng}
	\address[D.-Q. Deng]{Beijing Institute of Mathematical Sciences and Applications and Yau Mathematical Science Center, Tsinghua Univeristy, Beijing, People's Republic of China}
	\email{dingqun.deng@gmail.com}
	
	\author[R.-J. Duan]{Renjun Duan}
	\address[R.-J. Duan]{Department of Mathematics, The Chinese University of Hong Kong, Shatin, Hong Kong,
		People's Republic of China}
	\email{rjduan@math.cuhk.edu.hk}
	
	\date{\today}

\begin{abstract}
It has been unknown in kinetic theory whether the linearized Boltzmann or Landau equation with soft potentials admits a spectral gap in the spatially inhomogeneous setting. Most of existing works indicate a negative answer because the spectrum of two linearized self-adjoint collision operators is  accumulated to the origin in case of soft interactions. In the paper we rather prove it in an affirmative way when the space domain is bounded with an inflow boundary condition. The key strategy is to introduce a new Hilbert space with an exponential weight function that involves  the inner product of space and velocity variables and also has the strictly positive upper and lower bounds. The action of the transport operator on such space-velocity dependent weight function induces an extra non-degenerate relaxation dissipation in large velocity that can be employed to compensate the degenerate spectral gap and hence give the exponential decay for solutions in contrast with the sub-exponential decay in either the spatially homogeneous case or the case of torus domain. The result reveals a new insight of hypocoercivity for kinetic equations with soft potentials in the specified situation. 
\end{abstract}

\date{\today}
\maketitle	
\tableofcontents	

\thispagestyle{empty}

\section{Introduction}

\subsection{Models and motivations}

In the paper we are concerned with the formation of the {\it spectral gap} in case of soft potentials for two classical collisional kinetic equations, the Boltzmann equation with or without angular cutoff and the Landau equation, in the spatially inhomogeneous setting where the space domain is bounded. Let $L$ be the standard linearized collision operator around the normalized global Maxwellian $\mu=\mu(v):=(2\pi)^{-3/2}\exp(-|v|^2/2)$ with $v=(v_1,v_2,v_3)\in \R^3$, defined by
\begin{equation}
	\label{lco}
	Lf= \mu^{-1/2}Q(\mu^{1/2}f,\mu) +  \mu^{-1/2}Q(\mu,\mu^{1/2}f),
\end{equation}
where $Q$ is a bilinear collision operator to be specified later. Then, in terms of basic properties of $Q$, the linearized operator $L$ is self-adjoint and non-positive definite on $L^2(\R^3_v)$ with a five dimensional kernel space $\ker L={\rm span}\{\mu^{1/2},v_i\mu^{1/2}\,(1\le i\le 3),|v|^2\mu^{1/2}\}$. The long time behavior of solutions to the spatially homogeneous linear problem
\begin{equation}
	\label{shlp}
	\partial_t f=Lf,\quad t>0,\ v\in \R^3
\end{equation}
leads to the study of the eigenvalue problem:
\begin{equation}
	\label{def.ep}
	Lf=\lambda f
\end{equation}
for the operator $L$, see Cercignani \cite{Cercignani1988}. As is well known, this problem has five eigenfunctions for $\lambda=0$; all the other eigenvalues are negative. Moreover, thanks to Carleman \cite{Carleman1957}, Grad \cite{Grad1963}, Caflisch \cite{Caflisch1980} and many others, we know that in the hard potential case, there exists a spectral gap for $L$ so that an exponential decay to equilibrium as $t\to\infty$ can be obtained, while the spectrum in the soft potential case is continuously connected to the origin so that the spectral gap is lost and degenerate, see Figure \ref{fig1} below.

\begin{figure}[htbp]
	\centering
	\subfigure[Hard potentials]{
		\begin{minipage}[t]{0.25\linewidth}
			\centering
			\begin{figure}[H]
				\centering
				\begin{tikzpicture}[>=Stealth,scale=0.3]
					\draw[thick,->] (-6,0)--(4,0);
					\draw[thick,->] (0,-4)--(0,4);
					\foreach \Point in {(0,0), (-2.8,0), (-2.3,0),(-1.7,0),(-1,0)}{\fill[blue] \Point circle(2mm);}
					\draw[line width=3pt,blue,line cap=round] (-3.2,0)--(-6,0);
				\end{tikzpicture}
			\end{figure} 
		\end{minipage}%
	}%
	\hspace{3mm}
	\subfigure[Maxwell molecules]{
		\begin{minipage}[t]{0.25\linewidth}
			\centering
			\begin{figure}[H]
				\centering
				\begin{tikzpicture}[>=Stealth,scale=0.3]
					\draw[thick,->] (-6,0)--(4,0);
					\draw[thick,->] (0,-4)--(0,4);
					\foreach \Point in {(-5,0), (-4.4,0), (-3.6,0),(-2.6,0),(-1.4,0)}{\fill[blue] \Point circle(2mm);}
				\end{tikzpicture}
			\end{figure}  
		\end{minipage}%
	}%
	\hspace{3mm}
	\subfigure[Soft potentials]{
		\begin{minipage}[t]{0.25\linewidth}
			\centering
			\begin{figure}[H]
				\centering
				\begin{tikzpicture}[>=Stealth,scale=0.3]
					\draw[thick,->] (-6,0)--(4,0);
					\draw[thick,->] (0,-4)--(0,4);
					\foreach \Point in {(-5,0), (-4.2,0), (-3.5,0),(-2.9,0),(-2.4,0)}{\fill[blue] \Point circle(2mm);}
					\draw[line width=3pt,blue,line cap=round] (-2.1,0)--(0,0);
				\end{tikzpicture}
			\end{figure}    
		\end{minipage}%
	}%
	\centering
	\caption{Spectrum for $L$}
	\label{fig1}
\end{figure}

As a sufficient criterion, the existence of the spectral gap for $L$ by solving \eqref{def.ep} can be also verified from the coercivity estimate of $-L$ via the compactness argument or the constructive  approach. Indeed, it generally holds that there is a generic constant $\lambda_L>0$, depending only on the collision kernel of $Q$, such that 
\begin{equation}
	\label{coerv}
	(f,-Lf)_{L^2_v}\geq \lambda_L \int_{\R^3}\nu(v)|f|^2dv
\end{equation} 
for any $f=f(v)\in (\ker L)^\perp$ in $L^2(\R^3_v)$. This is the so-called {\it coercivity} estimate for $-L$ which has been extensively studied in \cite{Alexandre2000,Global2019,Alexandre2012,Baranger2005,Bobylev1988,Caflisch1980,Degond1997,Golse1986,Grad1963,Gressman2011,Guo2002a,Guo2003,Mouhot2006a,Mouhot2006,Mouhot2007,Pao1974,Chang1952,Chang1970} for different kinds of collision operators. Among them, we mention the fundamental works by Bobylev \cite{Bobylev1988} and Wang Chang and Uhlenbeck \cite{Chang1952,Chang1970} for the Maxwell molecules, the constructive proofs by Mouhot together with his collaborators in series of works \cite{Baranger2005,Mouhot2006a,Mouhot2006,Mouhot2007}, and the recent progress on the non-cutoff Boltzmann equation by Gressman and Strain \cite{Gressman2011} and AMUXY \cite{Alexandre2012}. The collision frequency function $\nu(v)$ in \eqref{coerv}, that turns out to be continuous and strictly positive at all $v\in\R^3$, has a strictly positive lower bound in the hard potential case, which then implies the exponential time decay of solutions to \eqref{shlp} in large time. However, in the soft potential case, $\nu(v)$ decays to zero as $|v|\to \infty$  and thus \eqref{coerv}  implies only the {\it degenerated spectral gap}, so that the solution to \eqref{shlp} may decay to zero only in a sub-exponential rate using the remarkable strategy by Caflisch \cite{Caflisch1980}. Here we also mention two very interesting recent works by Bobylev, Gamba and Potapenko \cite{Bobylev2015} and Bobylev, Gamba and Zhang \cite{Bobylev2017} about the sub-exponential rate relaxation for some collisional kinetic models.

In the spatially inhomogeneous setting, particularly when the space domain is a torus $\Omega=\T^3$, the spectral gap for the linearized equation
\begin{equation}
	\label{lpsi}
	\partial_t f+v\cdot \nabla_x f=Lf,\quad t>0,\ x\in \Omega,\ v\in \R^3,
\end{equation}
also has been extensively studied in the context of {\it hypocoercivity} that Villani \cite{Villani2009} first  introduced for investigating the interplay between the collision operator and the transport operator. For instance, in case of the Boltzmann collision with cutoff potentials, if one ignores the nonlocal integral part $K$ in  the Grad's splitting $L=-\nu(v)+K$ and further takes the Fourier transform of \eqref{lpsi} with respect to $x\in \T^3$, then   
the spectrum for the linear operator $-ik\cdot v-\nu(v)$, where $k\in\Z^3$ is the Fourier variable, can be described as in Figure \ref{fig2} below, see also Ellis and Pinsky \cite{Ellis1975}. Therefore, as in the spatially homogeneous case, it is hard to expect the spectral gap as well as the exponential decay for the linear operator $-v\cdot \nabla_x +L$ in case of soft potentials when the operator $K$ is included. In fact, for solutions to \eqref{lpsi} with $\Omega=\T^3$ in soft potential case, one can obtain only the sub-exponential or almost exponential decay for solutions when the initial data are weighted by an extra velocity function which is either exponential or polynomial, respectively; see many previous fundamental works, for instance, initiated by Caflish \cite{Caflisch1980a} and later extensively developed by Strain and Guo \cite{Strain2006a,Strain2007} as well as  Carrapatoso and Mischler \cite{Carrapatoso2017}.
\begin{figure}[htbp]
	\centering
	\subfigure[Hard potentials]{
		\begin{minipage}[t]{0.25\linewidth}
			\centering
			\begin{figure}[H]
				\centering
				\begin{tikzpicture}[>=Stealth,scale=0.3]
					\draw[thick,->] (-6,0)--(4,0);
					\draw[thick,->] (0,-4)--(0,4);
					\draw[dashed,thick,blue] (-2.5,-3.8)--(-2.5,3.8);
					\filldraw[blue] (-2.5,0) circle(2mm);
					\draw (-1.25,0) node[black,below]{\tiny $-\nu(0)$}; 
					\draw[rotate around={90:(-2.5,0)},color=blue,domain=-6:1,pattern=north east lines,pattern color=blue] plot (\x,1.25+\x+0.2*\x*\x);	
					\draw[rotate around={90:(-2.5,0)},<->,color=blue,domain=-6:1] plot (\x,1.25+\x+0.2*\x*\x);
				\end{tikzpicture}
			\end{figure} 
		\end{minipage}%
	}%
	\hspace{3mm}
	\subfigure[Maxwell molecules]{
		\begin{minipage}[t]{0.25\linewidth}
			\centering
			\begin{figure}[H]
				\centering
				\begin{tikzpicture}[>=Stealth,scale=0.3]
					\draw[thick,->] (-6,0)--(4,0);
					\draw[thick,->] (0,-4)--(0,4);
					\draw[thick,blue,<->] (-2.5,-3.8)--(-2.5,3.8);
					\filldraw[blue] (-2.5,0) node[black,below left]{\tiny $-\nu(0)$} circle(2mm);
				\end{tikzpicture}
			\end{figure}  
		\end{minipage}%
	}%
	\hspace{3mm}
	\subfigure[Soft potentials]{
		\begin{minipage}[t]{0.25\linewidth}
			\centering
			\begin{figure}[H]
				\centering
				\begin{tikzpicture}[smooth,>=Stealth,scale=0.3]
					\draw[thick,->] (-6,0)--(4,0);
					\draw[thick,->] (0,-4)--(0,4);
					\filldraw[blue] (-2.5,0) node[black,below left]{\tiny $-\nu(0)$} circle(2mm);
					\fill[rotate around={90:(0,0)},pattern color=blue,pattern=north east lines] (-3.6,0)--(-3.6,0.187) plot[domain=-3.6:3.6](\x,2.5*exp(-0.2*\x*\x)--(3.6,0.187)--(3.6,0)--(-3.6,0);
					\draw[rotate around={90:(0,0)},domain=-3.6:3.6,thick,blue,<->]plot (\x,2.5*exp(-0.2*\x*\x);
				\end{tikzpicture}
			\end{figure}     
		\end{minipage}%
	}%
	\centering
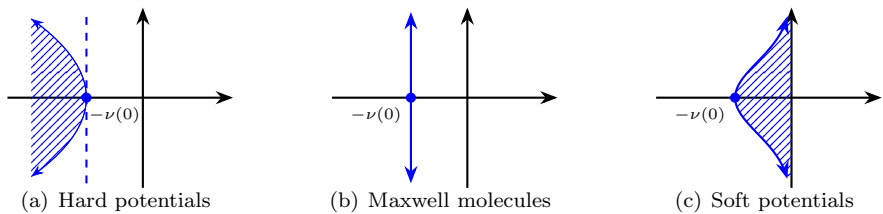
	\caption{Spectrum for $-ik\cdot v-\nu(v)$ with the parameter $k\in \Z^3$}
	\label{fig2}
\end{figure}

In contrast with either the spatially homogeneous case or the torus domain case mentioned above, we aim to recover in the paper the spectral gap as well as the exponential decay in case of soft potentials for general bounded domains with physical boundaries. 
For this, we consider $\Omega$ as a bounded open subset in $\R^3$ with the unit outward normal vector $n(x)$ which is assumed to exist almost everywhere on $\pa\Omega$. Then one can decompose the phase boundary $\pa\Omega\times \R^3$ as 
\begin{align*}
	\gamma_- &= \{(x,v)\in\partial\Omega\times\R^3 : n(x)\cdot v<0\},\quad\text{(the incoming set),}\\
	\gamma_+ &= \{(x,v)\in\partial\Omega\times\R^3 : n(x)\cdot v>0\},\quad\text{(the outgoing set),}\\
	\gamma_0 &= \{(x,v)\in\partial\Omega\times\R^3 : n(x)\cdot v=0\},\quad\text{(the grazing set).}
\end{align*}
With this notation, we consider the initial boundary value problem on the linearized equation \eqref{lpsi} supplemented with the {\it  inflow} boundary condition 
\begin{align*}
	f(t,x,v) = g(t,x,v),\quad \text{ on } [0,\infty)\times\gamma_-,
\end{align*}
for an appropriate function $g(t,x,v)$ decaying to zero exponentially. We will see that for the specified problem, the fact that the degenerate spectral gap can be refined to be non-degenerate and correspondingly the sub-exponential decay can be refined to be exponential is essentially a consequence of accounting for an extra effect by the transport operator in the bounded domain $\Omega$, namely, the action of the transport operator on a designed space-velocity dependent weight function induces an extra non-degenerate relaxation dissipation in large velocity.  This provides a new insight of hypocoercivity for kinetic equations with soft potentials under consideration.

In the meantime, as a byproduct of our refined weighted energy approach, in particular with the macro-micro decomposition introduced by Guo \cite{Guo2004,Guo2002a,Guo2003} as well as independently by Liu and Yu \cite{Liu2004a} and Liu, Yang and Yu \cite{Liu2004}, we will also establish the global well-posedness for the nonlinear Cauchy problem 
\begin{align}\label{e1}		\partial_tF+v\cdot\nabla_xF=Q(F,F), \quad F(0,x,v)=F_0(x,v),
\end{align} 
with inflow boundary condition
in the framework of perturbations in some new function spaces for the cutoff Boltzmann case.
Here, the non-negative $F(t,x,v)$ stands for the density distribution function of gas particles with velocity $v\in\R^3$ at time $t\ge 0$ and position $x\in\Omega$. For the perturbation formulation, setting $F(t,x,v) = \mu(v)+\mu^{1/2}(v)f(t,x,v)$, the problem \eqref{e1} can be rewritten as 
\begin{equation}\label{ncp}
	\partial_t f+v\cdot\nabla_xf = Lf +\Gamma(f,f),\quad f(0,x,v)=f_0(x,v),
\end{equation}
with inflow boundary condition
\begin{align*}
	f(t,x,v) = g(t,x,v),\quad \text{ on } \gamma_-,
\end{align*}
where $L$ is the linearized collision operator defined in \eqref{lco}  
and $\Gamma$ is the nonlinear collision operator given by 
\begin{equation*}
	\Gamma(f,g)=\mu^{-1/2}Q(\mu^{1/2}f,\mu^{1/2}g).
\end{equation*} 
For a complete literature review to the well-posedness and large time behavior of the Boltzmann or Landau equation, readers may refer to a recent work \cite{Duan2020} by the second author with Liu, Sakamoto and Strain and references therein.  Here we only mention that the pioneering work on the first global existence of solutions to the cutoff Boltzmann near global Maxwellians in the whole space was done by Ukai \cite{Ukai1974} for hard potentials and later by Ukai and Asano \cite{Ukai1982} for a certain range of soft potentials.  

In what follows we prescribe the bilinear collision operator $Q$ with only soft potentials under consideration. For the case of cutoff Boltzmann equation, the collision operator $Q$ is defined by 
\begin{align}\label{cutoffQ}
	Q(G,F) = \int_{\R^3}\int_{\mathbb{S}^{2}} B(v-v_*,\omega)\big[G(v'_*)F(v')-G(v_*)F(v)\big]\,d\omega dv_*.
\end{align} 
In this expression $v,v_*$ and $v',v'_*$ are velocities of a pair of particles before and after collision, given in terms of the $\omega$-representation by 
\begin{align*}
	v'=v+[(v_*-v)\cdot\omega]\omega,\quad v_*'=v_*-[(v_*-v)\cdot\omega]\omega,
\end{align*} 
satisfying the physical conservation laws of elastic collisions: $v+v_*=v'+v'_*$ and $|v|^2+|v_*|^2=|v'|^2+|v_*'|^2$. The Boltzmann collision kernel $B(v-v_*,\omega)$ for a monatomic gas is a non-negative function which only depends on relative velocity $|v-v_*|$ and the angle $\theta$ through $\cos\theta=\frac{v-v_*}{|v-v_*|}\cdot \omega$. We assume 
\begin{equation*}
	B(v-v_*,\omega)=|v-v_*|^\gamma b(\cos\theta)
\end{equation*}
with $-3<\gamma<0$ for soft potentials and the angular part $b(\cos\theta)$ satisfying the Grad angular cutoff assumption
$0<b(\cos\theta)\le C|\cos\theta|$.

For the case of non-cutoff Boltzmann equation, the collision operator $Q$ is defined by 
\begin{align*}
	Q(G,F) = \int_{\R^3}\int_{\mathbb{S}^{2}} B(v-v_*,\sigma)\big[G(v'_*)F(v')-G(v_*)F(v)\big]\,d\sigma dv_*.
\end{align*}  
In this expression $v,v_*$ and $v',v'_*$ are velocity pairs given in terms of the $\sigma$-representation by 
\begin{align*}
	v'=\frac{v+v_*}{2}+\frac{|v-v_*|}{2}\sigma,\quad v'_*=\frac{v+v_*}{2}-\frac{|v-v_*|}{2}\sigma,\quad \sigma\in\mathbb{S}^2,
\end{align*}
that again satisfy that 
$v+v_*=v'+v'_*$ and 
$|v|^2+|v_*|^2=|v'|^2+|v'_*|^2$.
The Boltzmann collision kernel $B(v-v_*,\sigma)$ depends only on $|v-v_*|$ and the deviation angle $\theta$ through $\cos\theta=\frac{v-v_*}{|v-v_*|}\cdot\sigma$. Without loss of generality we can assume $B(v-v_*,\sigma)$ is supported on $0\le\theta\le\pi/2$, since one can reduce the situation with {\it symmetrization}: $\overline{B}(v-v_*,\sigma)={B}(v-v_*,\sigma)+{B}(v-v_*,-\sigma)$. Moreover, we assume 
$B(v-v_*,\sigma) = |v-v_*|^\gamma b(\cos\theta)$,
where for the angular part, there exist $C_b>0$ and $0<s<1$ such that  
\begin{align*}
	\frac{1}{C_b\theta^{1+2s}}\le \sin\theta b(\cos\theta)\le \frac{C_b}{\theta^{1+2s}}, \quad\forall\,\theta\in (0,\frac{\pi}{2}].
\end{align*} 
We also assume $-3<\gamma<-2s$ for soft potentials.

For the case of Landau equation, the collision operator $Q$ is given by 
\begin{align*}
	Q(G,F)&=\nabla_v\cdot\int_{\R^3}\phi(v-v')\big[G(v')\nabla_vF(v)-F(v)\nabla_vG(v')\big]\,dv'\\
	&=\sum^3_{i,j=1}\partial_{v_i}\int_{\R^3}\phi^{ij}(v-v')\big[G(v')\partial_{v_j}F(v)-F(v)\partial_{v_j}G(v')\big]\,dv'.
\end{align*}
The non-negative definite matrix-valued function $\phi=[\phi^{ij}(v)]_{1\leq i,j\leq 3}$ takes the form of 
\begin{align*}
	\phi^{ij}(v) = \Big\{\delta_{ij}-\frac{v_iv_j}{|v|^2}\Big\}|v|^{\gamma+2},
\end{align*}
with $-3\leq \gamma<-2$ for soft potentials including the physically realistic Coulomb interactions.  

As a summary for {\it soft potentials} through the paper, we mean $-3<\gamma<0$ for cutoff Boltzmann collision operator, $-3<\gamma<-2s$ for non-cutoff Boltzmann collision operator and $-3\le \gamma<-2$ for Landau collision operator, and we also define $\kappa=\gamma$, $\gamma+2s$ and $\gamma+2$ in the respective case for later use. Note that $\kappa<0$ holds true. We may choose $\nu(v)\approx\langle v\rangle^\kappa$ for the collision frequency function $\nu(v)$ in the coercivity inequality \eqref{coerv} for the linearized collision operator $L$. 
%

\smallskip 

\subsection{Notations}


%
%
%
%

%

Before presenting the main results, we also specify some notations to be used through the paper.
Let $\<v\>=\sqrt{1+|v|^2}$ and 
$\partial^\alpha_\beta = \partial^{\alpha_1}_{x_1}\partial^{\alpha_2}_{x_2}\partial^{\alpha_3}_{x_3}\partial^{\beta_1}_{v_1}\partial^{\beta_2}_{v_2}\partial^{\beta_3}_{v_3}$,
where $\alpha=(\alpha_1,\alpha_2,\alpha_3)$ and $\beta=(\beta_1,\beta_2,\beta_3)$ are multi-indices. For $\beta$ and $\beta'$, if each component of $\beta'$ is not greater than that of $\beta$, we denote $\beta'\le\beta$. We will write $C>0$ to be a generic constant, usually large,  which may change from line to line. Denote the $L^2_v$ and $L^2_{x,v}$, respectively, as 
\begin{align*}
	|f|^2_{L^2_v} = \int_{\R^3}|f|^2\,dv,\quad \|f\|_{L^2_{x,v}}^2 = \int_{\Omega}|f|^2_{L^2_v}\,dx.
\end{align*}
For Landau equation, we denote 
\begin{align*}
	\sigma^{ij}(v)=\int_{\R^3}\phi^{ij}(v-v')\mu(v')\,dv',\quad \sigma^i(v)=\sum_{j=1}^3\int_{\R^3}\phi^{ij}(v-v')\frac{v'_j}{2}\mu(v')\,dv',
\end{align*}
and the dissipation norm as 
\begin{align*}
	|f|^2_{L^2_D}= \sum_{i,j=1}^3\int_{\R^3}\Big(\sigma^{ij}\partial_{v_i}f\partial_{v_j}f+\sigma^{ij}\frac{v_i}{2}\frac{v_j}{2}|f|^2\Big)\,dv,\quad
	\|f\|^2_{L^2_xL^2_D} = \int_{\Omega}|f|^2_{L^2_D}\,dx.
\end{align*}
For non-cutoff Boltzmann equation, as in \cite{Gressman2011}, we denote
\begin{align*}
	|f|^2_{L^2_D}&:=|\<v\>^{\frac{\gamma+2s}{2}}f|^2_{L^2_v}+ \int_{\R^3}dv\,\<v\>^{\gamma+2s+1}\int_{\R^3}dv'\,\frac{(f'-f)^2}{d(v,v')^{3+2s}}\1_{d(v,v')\le 1}. 
\end{align*}
We also denote $|f|^2_{L^2_{D,N}}=\sum_{|\beta|\le N}|\partial_\beta f|^2_{L^2_D}$ and $\|f\|^2_{L^2_xL^2_D} = \int_{\Omega}|f|_{L^2_D}^2\,dx$. The fractional differentiation effects are measured using the anisotropic metric on the {\it lifted} paraboloid
$d(v,v'):=\{|v-v'|^2+\frac{1}{4}(|v|^2-|v'|^2)^2\}^{1/2}$.
For cutoff Boltzmann equation, we denote 
$\nu$ as in \eqref{333a} by
\begin{equation*}
	\nu(v)=\int_{\R^3}\int_{\S^2}B(v-v_*,\omega)\mu(v_*)\,d\omega dv_*,
\end{equation*}
and hence,
$\nu(v)\approx \<v\>^\gamma$.
Correspondingly, we denote the dissipation norm as 
\begin{align}\label{defnu}
	|f|_{L^2_D} = |f|_{L^2_\nu} := \int_{\R^3}\nu(v)|f(v)|^2\,dv,\quad
	\|f\|^2_{L^2_xL^2_D} = \int_{\Omega}|f|^2_{L^2_D}\,dx.
\end{align}

The most crucial idea for recovering the spectral gap in case of soft potentials for the spatially inhomogeneous linear problem \eqref{lpsi} is to introduce a weight function in the phase variable $(x,v)\in \Omega\times \R^3$ that involves the inner product of $x$ and $v$. In fact, let $q>0$ be any positive constant, then we define the weight function 
\begin{align}\label{W}
	W = W(x,v) = \exp\Big(-q\frac{x\cdot v}{\<v\>}\Big).
\end{align}
It is straightforward to compute 
\begin{align}\label{W2}
	\partial_{x_i}W = -q\frac{v_i}{\<v\>}W,\quad \pa_{v_i}W = -q\big(\frac{x_i}{\<v\>}-\frac{v_ix\cdot v}{\<v\>^3}\big)W, 
\end{align}
and hence 
\begin{equation}\label{W1}
	-v\cdot\nabla_xW =q\frac{|v|^2}{\<v\>}W.
\end{equation}
Thus, by \eqref{W} and \eqref{W2} above, it holds that
\begin{align}\label{Wul}
	e^{-qC}\leq W\leq e^{qC},\quad
	|\partial_{x_i}W|\leq C qW,\quad |\partial_{v_i}W|\leq \frac{Cq}{\langle v\rangle}W,
\end{align}
for a generic constant $C\geq 1$ depending only on the size of $\Omega$ but not on $q$.
We remark that 
one may choose the weight function of the more general form 
\begin{equation}
	\label{17bb}
	W=\exp\left(-q\langle v\rangle^\vartheta (1-\epsilon \frac{x\cdot v}{\langle v\rangle})\right)
\end{equation} 
to generate higher velocity weight 
with parameters $q>0$, $0\leq\vartheta\leq 2$ and $\epsilon>0$. However, we would not pursue such more general formulation in order to keep the current discussions short and concise due to the fact that $W\approx 1$ for \eqref{W}.  



\subsection{Spectral gap}

%

In this subsection, we will present the spectral gap for Landau and non-cutoff Boltzmann equations with {\it zero inflow} boundary condition. For this, 
we consider $\Omega$ as a Lipschitz domain. That is, $\Omega\subset\R^3$ is an open subset and 
for each point $x_0\in\pa\Omega$, there exists $r>0$ and a Lipschitz function $\varphi:\R^2\to \R$, upon relabeling and reorienting the coordinates axes if necessary, such that 
\begin{align*}
	\Omega\cap B(x_0,r) = \big\{x\in B(x_0,r):x_3>\varphi(x_1,x_2)\big\}, 
\end{align*}
where $B(x_0,r)$ is the open ball with center $x_0$ and radius $r$. 
Then the unit outward normal  vector $n$ exists on $\pa\Omega$ almost everywhere. 
Therefore, we can consider linearized Boltzmann collision operator $L$ with 
 the {\it zero inflow} boundary condition for $f$:
\begin{align}\label{inflow}
	f(t,x,v) = 0,\quad \text{ on } \gamma_-.
\end{align}
It means the complete condensation state on the boundary:
\begin{align*}
	F(t,x,v) = \mu,\quad \text{ on }\gamma_-.
\end{align*}

To study the spectrum, we will split $L$ as $L=-A+K$ in Section \ref{sec2}, where $A$ is the dissipation part and $K$ is the relatively compact part. We define the linear operators 
\begin{align*}
	\Lambda = v\cdot \nabla_x - A, \\
	\mathcal{L} = v\cdot \nabla_x - L,
\end{align*}
and the inner product space 
\begin{align}\label{X1}
	X :=L^2(\Omega_x\times \R^3_v)
\end{align}
equipped with the inner product on $\Omega\times\R^3_v$:
\begin{align}\label{X}
	(f,h)_X = C_0(f,h)_{L^2_{x,v}} + (Wf,Wh)_{L^2_{x,v}},
\end{align}
for some large $C_0>0$ to be chosen large later in Lemma \ref{lem.coxv}. Let $\|\cdot\|_X$ be the norm induced by \eqref{X}. 
Thanks to \eqref{Wul}, we know that $\|\cdot\|_X$ is equivalent to  $\|\cdot\|_{L^2_{x,v}}$ and hence $(X,\|\cdot\|_X)$ is a Banach space. Note that $X$ has the same Banach topology as $L^2_{x,v}$. 
Here we denote $\mathcal{L}$ in the sense of distribution and write the definite domain of $\mathcal{L}$ to be 
\begin{align}\label{def.ddl}
	D(\mathcal{L}) = \{f\in X: f(t,x,v) = 0\ \text{ on } \gamma_-,\ \mathcal{L}f\in X\}. 
\end{align}
Note that $\mathcal{L}$ is not self-adjoint on $X$. Since $ D(\L)$ contains $C^\infty_c(\Omega\times\R^3)$, we know that $ D(\L)$ is dense in $X$ and $\L$ is densely defined. 
Recall that the kernel of $L$ is the span of $\{\mu^{1/2},v_i\mu^{1/2}$ $(1\le i\le 3),|v|^2\mu^{1/2}\}$. Then we denote $\P$ to be the projection onto $\ker L$:
\begin{align}\label{Pf}
	\P f = (a+b\cdot v+c|v|^2)\mu^{1/2}(v).
\end{align}

Before stating the result on spectrum we recall some notations from functional analysis. One may refer to \cite[Chap XIV]{Gohberg1990} for more details. 
Let $T$ be a linear operator on $X$ with  domain $D(T)\subset X$ and range $\Im\, T\subset X$. The resolvent set $\rho(T)$ is the set of all complex numbers $\lambda\in\C$ for which $\lambda-T$ is invertible and the inverse $(\lambda-T)^{-1}$ is bounded on $X$. The spectrum $\sigma(T)$ of $T$ is defined to be the complement in $\C$ of $\rho(T)$.  A point $\lambda_0\in\sigma(T)$ is called an eigenvalue of $T$ of finite type if $\lambda_0$ is an isolated point of $\sigma(T)$ and the associated projection 
$
P_{\lambda_0}=\frac{1}{2\pi i}\int_{\mathcal{C}}(\lambda -T)^{-1}\,d\lambda
$ 
has finite rank, where $\mathcal{C}$ is a contour around $\lambda_0$ separating $\lambda$ from the other eigenvalues. Since $\{\lambda_0\}=\sigma(T|_{\Im}P_{\lambda_0})$ and $\Im\, P_{\lambda_0}$ is finite dimension in $D(T)$, it follows that $\lambda_0$ is an eigenvalue of $T$. The linear operator $T$ is called a Fredholm operator if $T$ is closed and the integers
\begin{align*}
	\dim \ker T,\quad \dim(X/\Im\,{T})
\end{align*}
are finite. 
The {essential} spectrum of a linear operator $T$ on $X$ is defined to be the set of all $\lambda\in\C$ such that $\lambda-T$ is not a Fredholm operator. This set is denoted by $\sigma_{ess}(T)$.

The main result on the spectral gap for the linear operator $\mathcal{L}$ is stated as follows.

\begin{Thm}\label{Main1}
	Consider the Landau equation and non-cutoff Boltzmann equation. 
	The essential spectrum of $\mathcal{L} = -\Lambda+K$, as an operator on $(X,\|\cdot\|_X)$ with the definite domain $D(\mathcal{L})$ in \eqref{def.ddl}, lies in $\{z\in\C:\Re\, z\le -c_0\}$ for a constant $c_0>0$, and \begin{align}\label{gap}
		\sigma(\mathcal{L})\cap \{z\in\C:\Re\, z>-c_0\}\subset \{z\in\R:-c_0< z\le 0\}
	\end{align} contains only the discrete eigenvalues of finite type with possible accumulation points only on $\{\Re\, z=-c_0\}$. 
	Moreover, the kernel of $\mathcal{L}$ in $(X,\|\cdot\|_X)$ is $\{0\}$. 
\end{Thm}

We can regard $\Lambda$ as the dissipative part and $K$ as the relatively compact part. Theorem \ref{Main1} will be an application of stability for relatively compact perturbation for operators. Thus, we obtain the spectral gap between $\{\Re\, z=-c_0\}$ and $\{\Re\, z =0\}$ for soft potentials with only real isolated eigenvalues in it. It also shows that there's no other eigenvalue except $0$ in a neighborhood of the origin.

\subsection{Exponential decay for linear equation}

In this subsection, we consider the exponential time decay for the linearized Landau and Boltzmann equations with exponentially decay inflow boundary data. 
As explained before, the spectral gap estimate \eqref{gap} indicates the formation of the exponential decay of the linear equation \eqref{lpsi} for soft potentials in the spatially inhomogeneous setting, since $-c_0$ can be regarded as an upper bound of all nonzero eigenvalues. 


We consider the linear Cauchy problem 
\begin{align}\label{linear}
	\pa_tf+v\cdot\na_xf+Lf=0,\quad f(0,x,v)=f_0(x,v), 
\end{align} 
with 
the inflow boundary condition 
\begin{align}\label{2a}
	f(t,x,v) = g(t,x,v),\quad \text{ on } \gamma_-,
\end{align}
where $g(t,x,v)$ is assumed to decay to zero exponentially. 
 The key point is to observe that the linear dynamics weighted by $W$ in \eqref{W} behaves as 
\begin{equation}
	\label{exdiss}
	\partial_t (Wf) +v\cdot \nabla_x (Wf) +q\frac{|v|^2}{\langle v\rangle} Wf=WLf.
\end{equation}
Therefore, an extra relaxation dissipation $q\frac{|v|^2}{\langle v\rangle} Wf\approx q\frac{|v|^2}{\langle v\rangle}f$ that is non-degenerate in large velocity has been created and it can be employed to compensate the large-velocity degenerate dissipation of the form \eqref{coerv} for the operator $L$ in case of soft potentials. As such, the solutions decay in large time to zero in the $L^2_{x,v}$ norm with an exponential rate, which is then in the same situation as in the hard potential case. 
Precisely, via the weighted energy approach, we are able to prove the following

\begin{Thm}\label{MaindecayLandau}
	For any initial data $f_0\in  L^2_{x,v}$ and inflow boundary data $g\in L^2(\gamma_-)$ satisfying 
	\begin{align}\label{g1}
	\|f_0\|^2_{L^2_{x,v}}+\sup_{t\ge 0}\int_{\pa\Omega}\int_{v\cdot n<0}|v\cdot n|e^{\delta_0 t}|g|^2\,dvdS(x)<\infty,
	\end{align}
for some $\delta_0>0$,
	 there exists $0<\lam<\delta_0$ such that  the linearized Landau or Boltzmann equation (including cutoff and non-cutoff cases) \eqref{linear}  with inflow boundary condition \eqref{2a} admits a unique solution $f$ satisfying 
	\begin{align*}
		\sup_{t\ge 0}\{e^{\lam t}\|f(t)\|_{L^2_{x,v}}^2\} \lesssim \|f_0\|^2_{L^2_{x,v}} + \sup_{t\geq 0}\int_{\pa\Omega}\int_{v\cdot n<0}|v\cdot n|e^{\delta_0 t}|g|^2\,dvdS(x).
	\end{align*}
\end{Thm}

\subsection{Exponential decay for nonlinear equation}

In this subsection, we will give the exponential time decay of solutions to the initial boundary value problem on the nonlinear Boltzmann equation with cutoff potentials
\begin{align}\label{1}
	\partial_tf  + v\cdot\nabla_xf  =  Lf + \Gamma(f,f),\quad f(0,x,v)=f_0(x,v), 
\end{align}
with a general inflow boundary condition
\begin{align}\label{2}
	f(t,x,v) = g(t,x,v),\quad \text{ on } \gamma_-,
\end{align}
where $g(t,x,v)$ is assumed to decay to zero exponentially. 
The linearized homogeneous form of \eqref{1} is 
\begin{align}
	\label{1a}
	\partial_tf  + v\cdot\nabla_xf  =  Lf. 
\end{align}
Here we let the domain $\Omega\subset\R^3$ be a smooth 
connected open set such that 
\begin{align}\label{Omega2}
	\Omega = \{x\in\R^3:\xi(x)<0\}
\end{align} with $\xi(x)$ being a smooth function. We assume $\na_x\xi(x)\neq 0$ at the boundary $\xi(x)=0$. Then the unit outward normal vector at $\pa\Omega$ is given by 
	$n(x) = 
	\na_x\xi(x)/|\na_x\xi(x)|$,
and it can be extended smoothly near $\pa\Omega=\{x\in\R^3:\xi(x)=0\}$. 
We call $\Omega$ as strictly convex if there exists $c_\xi>0$ such that 
\begin{align}
	\label{convex}
	\sum_{i,j=1}^3\pa_{x_ix_j}\xi(x)\zeta^i\zeta^j\ge c_\xi|\zeta|^2
\end{align}
for all $x\in \overline\Omega$ and all $\zeta\in\R^3$.

Denote a velocity weight function for $\rho>0$ and $\beta\in\R$:
\begin{align}\label{w}
	w(v) = (1+\rho^2|v|^2)^\beta. 
\end{align}
Then we can reformulate  the Boltzmann equation \eqref{ncp} by using the weighted function $h:=wWf$:
\begin{equation}\label{cutoff}
	\left\{\begin{aligned}
		&\{\pa_t+v\cdot\na_x+q{|v|^2}\<v\>^{-1}+\nu\}h=wW\Gamma\Big(\frac{h}{wW},\frac{h}{wW}\Big),\\
		&h(0,x,v)=wWf_0(x,v),\quad h|_{\gamma_-}=wWg. 
	\end{aligned}\right.
\end{equation} 

Given $(t,x,v)$, let $[X(s),V(s)]=[X(s;t,x,v),V(s;t,x,v)]=[x+(s-t)v,v]$ be the trajectory (or the characteristic) for the cutoff Boltzmann equation such that
\begin{align*}
	\frac{dX(s)}{ds}=V(s),\quad\frac{dV(s)}{ds}=0
\end{align*}
with the initial condition that $[X(t;t,x,v),V(t;t,x,v)]=[x,v]$.
\begin{Def}
	[Backward exit time] For $(x,v)$ with $x\in\ol\Omega$ such that there exists some $\tau>0$, $x-sv\in\Omega$ for $0\le s\le \tau$, we define $t_{\b}>0$ to be the last moment at which the back-time straight line $[X(s;0,x,v),V(s;0,x,v)]$ remains in the interior of $\Omega$:
	\begin{align}\label{backward}
		t_\b(x,v)=\sup\{\tau>0\,:\, x-sv\in\Omega \text{ for all }0\le s\le \tau\}.
	\end{align}
\end{Def}
Clearly, for any $x\in\Omega$, $t_\b(x,v)$ is well-defined for all $v\in\R^3$. If $x\in\pa\Omega$, $t_\b(x,v)$ is well-defined for all $v\cdot n(x)>0$. For any $(x,v)$, we use $t_\b(x,v)$ whenever it is well-defined. We have $x-t_\b v\in\pa\Omega$ and $\xi(x-t_\b v)=0$. We also define 
\begin{align*}
	x_{\b}(x,v)=x(t_\b) = x-t_\b v\in\pa\Omega.
\end{align*}
Then we always have $v\cdot n(x_\b)\le 0$. 

Our result on the exponential decay for the cutoff Boltzmann equation is given as follows. 

\begin{Thm}\label{Maincutoff}
	Let $\Omega$ be given in \eqref{Omega2} and fix $q>0$.
	Assume $w^{-2}\<v\>\in L^1(\R^3)$ in \eqref{w}. There exists $\delta>0$ such that if $F_0(x,v)=\mu+\sqrt\mu f_0(x,v)\ge 0$ and 
	\begin{align}
		\label{prioricutoff}
		\|wWf_0\|_{L^\infty_{x,v}}+\sup_{t\ge 0}e^{\lam_0 t}\|wWg(t)\|_{L^\infty(\gamma_-\times\R^3)}\le \delta, 
	\end{align}
with some constant $\lam_0>0$, then there exists a unique solution $F(t,x,v)=\mu+\sqrt\mu f(t,x,v)\ge 0$ to the inflow problem  \eqref{1} and \eqref{2} on the cutoff Boltzmann equation \eqref{cutoffQ}, satisfying the estimate that there exists $0<\lam<\lam_0$ 
such that 
\begin{align*}
	\sup_{t\ge 0}e^{\lam t}\|wWf(t)\|_{L^\infty_{x,v}}\lesssim 
	\|wWf_0\|_{L^\infty_{x,v}}+\sup_{s\ge 0}e^{\lam_0s}\|wWg(s)\|_{L^\infty(\gamma_-)}.
\end{align*}
Moreover, if $\Omega$ is strictly convex in the sense of \eqref{convex}, $f_0(x,v)$ is continuous on $\overline\Omega\times\R^3\setminus\gamma_0$, and $g$ is continuous in $[0,\infty)\times\{\pa\Omega\times\R^3\setminus\gamma_0\}$ with 
	$f_0(x,v)=g(0,x,v)$
	on $\gamma_-$,
then $f(t,x,v)$ is continuous in $[0,\infty)\times\{\overline\Omega\times\R^3\setminus\gamma_0\}$. 

\end{Thm}

Noticing $W\geq 1/ C_\Omega$ by \eqref{Wul}, we know that the smallness assumption \eqref{prioricutoff} can be replaced by $\|wf_0\|_{L^\infty_{x,v}}+\sup_{t\ge 0}e^{\lam_0 t}\|wg(t)\|_{L^\infty(\gamma_-\times\R^3)}\le C_\Omega\delta$.
\smallskip

The above result gives the global-in-time existence and large-time behavior for small-amplitude classical solutions to cutoff Boltzmann equation with/without cutoff in the space $L^\infty_{x,v}$ for the case of soft potentials. In particular, we finally recover the exponential decay $e^{-\lambda t}$ as $t\to \infty$ with an explicitly constructive constant $\lambda>0$, which is an essential improvement to the sub-exponential decay $e^{-\lambda t^p}$ with a parameter $0<p<1$ depending only on the collision kernel  in \cite{Caflisch1980a,Gressman2011,Strain2007} where they also have imposed an extra exponential velocity decay on initial data. In the current work, we obtain the exponential decay of solutions in the function space 
that is equivalent to $L^\infty_{x,v}$ without any velocity weight lost. 

In the end we conclude this section with comments on some possible applications of our new approach as well as prospective problems that we think are important to be further explored in the future. In particular, we are able to employ the space-velocity weight function $W=W(x\cdot v,v)$ of the simple form \eqref{W} or of the more general form \eqref{17bb} to create an extra dissipation even for the transport equation as in \eqref{exdiss}. Thus, we believe that our general method can be widely applied in kinetic problems with the transport term. First of all, it is interesting to see if it is applicable to the cutoff Boltzmann equation in general bounded domains supplemented with other types of boundary conditions in the $L^2\cap L^\infty$ framework  developed first by Guo \cite{Guo2009} and extensively studied later by many people, for instance, whether the results for soft potentials in \cite{Liu2016} and \cite{Duan2019a} could be refined using our weight function. Moreover, it is possible to apply the current approach to study the non-cutoff Boltzmann equation, Landau equation and related models with self-consistent forces in case of soft potentials for the problems in \cite{Duan2020,Guo2012,Guo2020}. A more challenging problem proposed in an original work \cite{Gualdani2017} is to extend their perturbation theory with polynomial tails to soft potentials. In this direction we mention recent great progresses for the Landau equation \cite{Carrapatoso2017} and the non-cutoff Boltzmann equation \cite{Alonso2020} and \cite{Alonso2020a} in torus.

The rest of the paper is arranged as follows. In section \ref{sec2}, we split the linearized collision operator and give some basic estimates. In section \ref{sec3}, we recover the spectrum gap for soft potential and prove Theorem \ref{Main1} by using relatively compact perturbation. In section \ref{sec4}, we analyze the linear problem \eqref{linear} and \eqref{2a} to give the proof of Theorem \ref{MaindecayLandau}.  In section \ref{sec5}, we study the nonlinear problem \eqref{1} and \eqref{2} with weight $W$ involved for the cutoff Boltzmann equation. Exponential time decay of solutions can be observed for soft potential.

\section{Basic estimates}\label{sec2}

In this section we list some basic estimates on the linear and nonlinear collision operators in the Landau and Boltzmann cases specified in the previous section.

\begin{Lem}\label{Lem21}
	Let $R>0$ be arbitrary. Assume $\gamma\ge -3$ for Landau case, $\gamma>\max\{-3,-\frac{3}{2}-2s\}$ for non-cutoff Boltzmann case and $\gamma>-3$ for cutoff Boltzmann case, we split $L=-A+K$ with respect to $R$ (the decomposition depends on $R$) such that 
	\begin{align}\label{16}
		(A f,f)_{L^2_{v}} &\approx |f|^2_{L^2_{D}},\\
		|(A f,g)_{L^2_{v}}| &\lesssim  |f|_{L^2_{D}}|g|_{L^2_D}.\label{16a}
	\end{align}
	There exist $c_1,C_1>0$ such that 
\begin{align}\label{e21}
	(-Lf,f)_{L^2_v}\ge c_1|\{\I-\P\}f|_{L^2_D}^2,
\end{align}
and 
\begin{align}\label{e22}
	|(W^2Af,f)_{L^2_v}|+|(W^2Lf,f)_{L^2_v}|\le C_1|f|_{L^2_D}^2. 
\end{align}
For Landau and non-cutoff Boltzmann case, we further have 
\begin{align}\label{17}
	|WKf|_{L^2_v}\lesssim |\chi_Rf|_{L^2_v}, 
\end{align}
for $q\ge 0$ in \eqref{W}. Here $\chi_R$ is given by \eqref{def.chieps}. 
	
\end{Lem}

\begin{proof}
\noindent{\bf Case I: Landau equation.}
We will apply the decomposition $L=-A+K$ in \cite[Section 4.2]{Yang2016}. Let $\varepsilon,R>0$ and choose  smooth cutoff functions $\chi_\ve(|v|),\chi_R(|v|)\in[0,1]$ such that 
\begin{equation}
	\label{def.chieps}
	\chi_R(|v|)=1\text{ if } |v|<R;\  \chi_R(|v|)=0 \text{ if } |v|>2R.
\end{equation}
Note that by \cite[Lemma 1]{Guo2002a}, we have 
\begin{align*}
	\Gamma(f,\mu^{1/2})
	&= -\mu^{-1/2}\partial_{v_i}\Big\{\mu\Big[\phi^{ij}*\Big(\mu\partial_{v_j}\big[\mu^{-1/2}f\big]\Big)\Big]\Big\}.
\end{align*}
Here and below the repeated indices are implicitly summed over.  Then we can split $L=-A+K$ with   
\begin{align}\label{18}
	-Af &= \partial_{v_i}(\sigma^{ij}\partial_jf) - \sigma^{ij}\frac{v_i}{2}\frac{v_j}{2}f
	+\partial_{v_i}\sigma^i(1-\chi_R)f+A_1f\notag\\
	&\qquad+(K_1-\chi_RK_1\chi_R)f,\\
	Kf &= \partial_{v_i}\sigma^i\chi_Rf + \chi_RK_1\chi_Rf,\label{18a}
\end{align}
and $A_1$ and $K_1$ are respectively given by 
\begin{align*}
	A_1f &= -\mu^{-1/2}\partial_{v_i}\Big\{\mu\Big[\Big(\phi^{ij}\chi_\ve\Big)*\Big(\mu\partial_{v_j}\big[\mu^{-1/2}f\big]\Big)\Big]\Big\},\\
	K_1f &=  -\mu^{-1/2}\partial_{v_i}\Big\{\mu\Big[\Big(\phi^{ij}\big(1-\chi_\ve\big)\Big)*\Big(\mu\partial_{v_j}\big[\mu^{-1/2}f\big]\Big)\Big]\Big\},
\end{align*}
with the convolution taken with respect to the velocity variable $v$. 
From \cite[Lemma 3]{Guo2002a}, we know that  
\begin{align}\label{65a}
	|\partial_\beta\sigma^{ij}(v)|+|\partial_\beta\sigma^i(v)|\le C_\beta(1+|v|)^{\gamma+2-|\beta|},
\end{align}and 
\begin{align}\label{65}
	&\sigma^{ij}g_i\overline{g_j}=\lambda_1(v)
	|P_vg|^2+\lambda_2(v)
	|(I-P_v)g|^2,\\
	&\sigma^{ij}\frac{v_i}{2}\frac{v_j}{2}|g|^2=\lambda_1(v)|v|^2|g|^2,\label{66}
\end{align}
where $P_v$ is the projection along the direction $v$ defined as
	$P_vg = \frac{v}{|v|}\sum^3_{i=1}\frac{v_i}{|v|}g_i$.
Moreover, it holds that $\lambda_1(v),\lambda_2(v)>0$ both are strictly positive for all $v\in \R^3$ and that there are constants $c_1,c_2>0$ such that as $|v|\to\infty$,
\begin{align*}
	\lambda_1(v)\sim c_1(1+|v|)^\gamma,\quad\lambda_2(v)\sim c_2(1+|v|)^{\gamma+2}.
\end{align*}	
Also, \cite[Corollary 1]{Guo2002a} and \cite[Lemma 5]{Strain2007} show that 
\begin{align}\label{27a}
	|g|^2_{L^2_D} \approx |\<v\>^{\frac{\gamma}{2}}P_v\nabla_vg|^2_{L^2_v}+|\<v\>^{\frac{\gamma+2}{2}}(I-P_v)\nabla_vg|^2_{L^2_v}+|\<v\>^{\frac{\gamma+2}{2}}g|^2_{L^2_v}.
\end{align}
By \eqref{65}, \eqref{66} and \eqref{18}, we have 
\begin{align}\label{20a}
	\big|(A f,f)_{L^2_{x,v}}- \|f\|^2_{L^2_xL^2_{D}} \big|\le 
	 \big|(\partial_{v_i}\sigma^i(1-\chi_R)f+A_1f+(K_1-\chi_RK_1\chi_R)f,f)_{L^2_{x,v}}\big|.
\end{align}
By \eqref{65a}, $|\partial_{v_i}\sigma^i|\lesssim \<v\>^{\gamma+1}$ and hence, 
\begin{align}\label{20d}
	|(W^2\partial_{v_i}\sigma^i(1-\chi_R)f,f)_{L^2_v}|&\lesssim R^{-1}|\<v\>^{\frac{\gamma+2}{2}}f|^2_{L^2_v}\le R^{-1}|f|^2_{L^2_D}. 
\end{align}
Also, noticing $W\lesssim 1$ and taking integration by parts with respect to $v_i$, we have 
\begin{align}
	&|(WA_1f,Wf)_{L^2_v}|\notag\\
	&= \Big|\int_{\R^3}\int_{\R^3}\partial_{v_i}\big(W^2\mu^{-1/2}(v)f(v)\big)\phi^{ij}(v-v')\chi_\ve(|v-v'|)\notag\\
	&\qquad\qquad\qquad\qquad\qquad\times\mu(v)\mu(v')\partial_{v_j'}\big[\mu^{-1/2}f\big](v')\,dv'dv\Big|\notag\\
	&\le\notag \Big(\int_{\R^3}\int_{\R^3}|\phi^{ij}(v-v')|^2\chi_\ve(|v-v'|)\mu^{\frac{1}{2}}(v)\mu^{\frac{1}{2}}(v')\,dv'dv\Big)^{1/2}\notag\\
	&\notag\qquad\times\Big(\int_{\R^3}\int_{\R^3}\big|\partial_{v_i}\big(W^2\mu^{-\frac{1}{2}}f\big)(v)\big|^2\mu^{\frac{3}{2}}(v)\mu^{\frac{3}{2}}(v')\big|\partial_{v_j'}\big[\mu^{-\frac{1}{2}}f\big]\big|^2\,dv'dv\Big)^{1/2}\notag\\
	&\lesssim \varepsilon^{2\gamma+7}|f|^2_{L^2_D}, \label{14}
\end{align}
where we recall \eqref{def.chieps} for the definition of $\chi_\ve(\cdot)$ depending on $\varepsilon$.  
Note that this estimate holds for both $q=0$ and $q>0$.
Similarly, for the part $1-\chi_\ve$, there's no singularity near $v-v'=0$ and we can put the derivative to $\phi^{ij}(v-v')(1-\chi_\ve)(|v-v'|)\mu(v)\mu(v')$ to obtain 
\begin{align}
	&|(WK_1f,Wg)_{L^2_v}|\notag\\
	&= \Big|\int_{\R^3}\int_{\R^3}\partial_{v_i}\big(W^2\mu^{-1/2}g\big)(v)\phi^{ij}(v-v')(1-\chi_\ve)(|v-v'|)\notag\\
	&\qquad\qquad\qquad\qquad\qquad\times\mu(v)\mu(v')\partial_{v_j'}\big[\mu^{-1/2}f\big](v')\,dv'dv\Big|\notag\\
	&\lesssim\notag C_\varepsilon\int_{\R^3}\int_{\R^3}\big|\big(W^2\mu^{-1/2}g\big)(v)\big|\mu^{2/3}(v)\mu^{2/3}(v')\big|\big[\mu^{-1/2}f\big](v')\big|\,dv'dv\notag\\
	&\lesssim\notag C_\varepsilon\int_{\R^3}\int_{\R^3}|g(v)|\mu^{1/6}(v)\mu^{1/6}(v')|f(v')|\,dv'dv\notag\\
	&\lesssim C_\varepsilon|\mu^{1/10}f|_{L^2_v}|\mu^{1/10}g|_{L^2_v}.
	\label{15}
\end{align}
Thus for $q\ge 0$ in \eqref{W}, 
\begin{align}\label{15a}
	|(W(K_1-\chi_RK_1\chi_R)f,Wf)_{L^2_v}|&\lesssim e^{-R^2/40}|\mu^{1/20}f|^2_{L^2_v}\lesssim e^{-R^2/40}|f|^2_{L^2_D}.
\end{align}
Noticing $W=e^{-q\frac{x\cdot v}{\<v\>}}\approx 1$, if we choose $q=0$ in \eqref{14} and \eqref{15a}, $R>0$ large and $\varepsilon>0$ small enough, then \eqref{20a} gives 
\begin{align*}
	(A f,f)_{L^2_{v}} \approx |f|^2_{L^2_{D}}.
\end{align*}
For $q\ge 0$ in \eqref{W}, we have 
\begin{align}\label{15d}
	|WKf|_{L^2_v}\lesssim |\chi_Rf|_{L^2_v}. 
\end{align}
Also, using \cite[Proposition 1]{Strain2013}, i.e. $|(\Gamma(f,g),h)_{L^2_v}|\lesssim |f|_{L^2_v}|g|_{L^2_D}|h|_{L^2_D}$, \eqref{14} and \eqref{15}, we have 
\begin{align*}
	|(Af,g)_{L^2_v}|\lesssim |f|_{L^2_D}|g|_{L^2_D}. 
\end{align*}
 To prove \eqref{e22}, we only need to consider the first two terms in \eqref{18}, since the other three terms have been discussed in \eqref{20d}, \eqref{14}, \eqref{15} and \eqref{15a}. To the end, we write $\pa_{v_i}=\pa_i$ with $1\leq i\leq 3$ for brevity. 
 We obtain from \eqref{65} and \eqref{66} that 
 \begin{align}\label{20b}\notag
 	&\quad\,(W\partial_i(\sigma^{ij}\partial_jf),Wf)_{L^2_{x,v}} - (W\sigma^{ij}\frac{v_i}{2}\frac{v_j}{2}f,Wf)_{L^2_{x,v}}\\
 	&=\notag -(\partial_i(W^2)\sigma^{ij}\partial_jf,f)_{L^2_{x,v}} - (W^2\sigma^{ij}\partial_jf,\partial_if)_{L^2_{x,v}} - (W\sigma^{ij}\frac{v_i}{2}\frac{v_j}{2}f,Wf)_{L^2_{x,v}}\\
 	&\le -\|f\|_{L^2_xL^2_{D}} - (\partial_i(W^2)\sigma^{ij}\partial_jf,f)_{L^2_{x,v}},
 \end{align}
 since $W\gtrsim 1$. 
 For the second right-hand term, using integration by parts with respect to $\partial_j$, one has 
 \begin{align*}
 	(\partial_i(W^2)\sigma^{ij}\partial_jf,f)_{L^2_{x,v}} 
 	&= -(\partial_{i}(W^2)\sigma^{ij}f,\partial_j f)_{L^2_{x,v}} - (\partial_{ij}(W^2)\sigma^{ij}f, f)_{L^2_{x,v}} \\
 	&\quad-(\partial_{i}(W^2)\partial_j\sigma^{ij}f, f)_{L^2_{x,v}}.
 \end{align*}
 Noticing that the first right-hand term is the same as the left-hand side, we have 
 \begin{align}\label{20c}\notag
 	(\partial_i(W^2)\sigma^{ij}\partial_jf,f)_{L^2_{x,v}} &= -\frac{1}{2} (\partial_{ij}(W^2)\sigma^{ij}f, f)_{L^2_{x,v}} -\frac{1}{2}(\partial_{i}(W^2)\partial_j\sigma^{ij}f, f)_{L^2_{x,v}}\\
 	&\lesssim \|\<v\>^{\frac{\gamma+2}{2}}f\|^2_{L^2_{x,v}}\lesssim \|f\|^2_{L^2_xL^2_D},
 \end{align}
 where we have used the fact that $|\partial_i(W^2)|\lesssim q W^2$ and $|\partial_{ij}(W^2)|\lesssim q W^2$ thanks to \eqref{W2}, as well as \eqref{65a}. 
Collecting the estimates \eqref{20d}, \eqref{14}, \eqref{15}, \eqref{15a}, \eqref{20b} and \eqref{20c}, we have 
\begin{align}\notag
	(W^2Af,f)_{L^2_{x,v}}\lesssim \|f\|_{L^2_xL^2_D}^2.
\end{align}
Together with \eqref{15d}, we obtain \eqref{e22}.

\medskip
\noindent{\bf Case II: Non-cutoff Boltzmann equation.}
We will use Pao's splitting as $\tilde{{\nu}}(v)={\nu}(v)+{\nu}_K(v)$; cf. \cite[pp. 568 eq. (65), (66)]{Pao1974} and \cite{Gressman2011}. Then $\nu(v)\approx\<v\>^{\gamma+2s}$ and $|{\nu}_K(v)|\lesssim \<v\>^\gamma$. We split $L=-A+K$ with 
\begin{align}\notag
	-Af &= \Gamma(\mu^{1/2},f)-{\nu}_K(v)f+\nu_K(v)(1-\chi_R)f\\&\notag\qquad + (\Gamma(f,\mu^{1/2})-\chi_R\Gamma(\chi_Rf,\mu^{1/2})),\\
	Kf &= \chi_R\Gamma(\chi_Rf,\mu^{1/2})+{\nu}_K(v)\chi_Rf.\label{18b}
\end{align}
Here $R>0$ will be chosen large later. 
Recall that $W\approx 1$, so such weight actually acts like a constant. Thus, one can apply the same proof of \cite[Lemma 2.5]{Gressman2011} to deduce that 
\begin{equation}
	(-\Gamma(\mu^{1/2},f)+{\nu}_K(v)f,f)_{L^2_v}\approx |f|^2_{L^2_D}\label{10aa}.
\end{equation}
According to \cite[Lemma 2.15]{Alexandre2012}, we have 
\begin{align*}
	|(\Gamma(f,\mu^{1/2}),g)_{L^2_v}|\lesssim |\mu^{1/10^3}f|_{L^2_v}|\mu^{1/10^3}g|_{L^2_v},
\end{align*}
and hence 
\begin{align}\label{212}
	(\Gamma(f,\mu^{1/2})-\chi_R\Gamma(\chi_Rf,\mu^{1/2}),f)_{L^2_v}\lesssim e^{-\frac{(2R)^2}{2000}}\|f\|_{L^2_D}^2.
\end{align}
In addition, 
\begin{align}\label{212a}
	|({\nu}_K(v)\chi_Rf,f)_{L^2_v}|\lesssim R^{-2s}|\<v\>^{\frac{\gamma+2s}{2}}f\|_{L^2_v}^2\lesssim R^{-2s}|f|^2_{L^2_D}. 
\end{align}
Combining \eqref{10aa}, \eqref{212} and \eqref{212a}, we have 
\begin{align*}
	(Af,f)_{L^2_v}\approx |f|_{L^2_D}^2. 
\end{align*}
Noticing $W\lesssim 1$, one can apply the same calculation as \cite[Lemma 2.4]{Gressman2011} to obtain 
\begin{equation}
	|(WAf,Wf)_{L^2_v}|\lesssim |f|^2_{L^2_D}.\label{10a}
\end{equation}
On the other hand, by \cite[eq. (6.6), pp.  817]{Gressman2011} i.e $|(\Gamma(f,g),h)_{L^2_v}|\lesssim |f|_{L^2_v}|g|_{L^2_D}|h|_{L^2_D}$, \eqref{212} and \eqref{212a}, we have 
\begin{equation*}
	|(Af,g)_{L^2_v}|\lesssim |f|_{L^2_D}|g|_{L^2_D}.
\end{equation*}
Noticing $W\lesssim 1$, we have 
\begin{align}\label{10d}
	|WKf|_{L^2_v}\lesssim |\chi_Rf|_{L^2_v}, 
\end{align}
for $q\ge 0$ in \eqref{W}. 
The estimate \eqref{e22} follows from \eqref{10a} and \eqref{10d}.
Also, \cite[eq. (2.15)]{Gressman2011} shows that 
\begin{align}\label{27b}
	|\<v\>^{\frac{\gamma+2s}{2}}g|^2_{L^2_v}+|\<v\>^{\frac{\gamma}{2}}\<D_v\>^sg|^2_{L^2_v}\lesssim |g|^2_{L^2_D}\lesssim |\<v\>^{\frac{\gamma+2s}{2}}\<D_v\>^sg|^2_{L^2_v}.
\end{align}

\medskip
\noindent{\bf Case III: Cutoff Boltzmann equation.}
As in \cite{Guo2003}, we split $L=-A+K$ with 
\begin{align}\label{333a}
	Af &= {\nu}(v)f = f(v)\int_{\R^3}\int_{\S^2}B(v-v_*,\omega)\mu(v_*)\,d\omega dv_*,\\
	Kf &=\notag \int_{\R^3}\int_{\S^2}B(v-v_*,\omega)\mu^{1/2}(v_*)\\	&\qquad\quad\times\Big(\mu^{1/2}(v'_*)f(v')+\mu^{1/2}(v')f(v'_*)-\mu^{1/2}(v)f(v_*)\Big)\,d\omega dv_*.\notag
\end{align}
By \cite[Lemma 1]{Guo2003}, we have 
\begin{align*}
	|(Kf,g)_{L^2_v}|\le C|f|_{L^2_D}|g|_{L^2_D}, 
\end{align*}
and hence, by definition \eqref{defnu}, 
\begin{align}\label{nu2}
	|\nu^{-1/2}Kf|_{L^2_v}\le C|\nu^{1/2}f|_{L^2_v}. 
\end{align}
Consequently, noticing $W\lesssim 1$, we have 
\begin{align*}
	|(W^2Af,f)_{L^2_v}|+|(W^2Lf,f)_{L^2_v}|\le C |\nu^{1/2}f|_{L^2_v}. 
\end{align*}

\medskip 
In the end, we point out that for both Landau and Boltzmann cases, it follows from \cite[Lemma 5]{Guo2002a}, \cite[Theorem 8.1]{Gressman2011} and \cite[Lemma 3]{Guo2003} that 
\begin{align*}
	(-Lf,f)_{L^2_v} \ge c_1 |\{\I-\P\}f|^2_{L^2_D},
\end{align*}
with a generic constant $c_1>0$.
%
\end{proof}

\section{Spectral gap}\label{sec3}

In this section, we are going to show Theorem \ref{Main1} on the formation of the spectral gap for the operator $\L=-v\cdot\nabla_x +L$ in the space $X=L^2_{x,v}$ equipped with the inner product $(\cdot,\cdot)_X$ as in \eqref{X} where the constant $C_0$ is chosen to be large. 

	Split $L=-A+K$ as in Section \ref{sec2} and  consider operators 
\begin{align*}
	\Lambda f:= 	v\cdot\na_x f+A f \ \text{ and } \ Kf.
\end{align*} 
We start with the following

\begin{Lem}\label{lem.coxv}
	Let $q>0$ in \eqref{W}. 
	For Landau case and Boltzmann case, there are constants $C_0,c_0,\lam>0$ such that 
	\begin{align}\label{10}
		(\Lambda f,f)_X \ge 
		\frac{C_0\lam}{2}\|f\|^2_{L^2_xL^2_{D}} + q\|\<v\>^{-\frac{1}{2}}|v|Wf\|^2_{L^2_{x,v}}
		\ge c_0\|\<v\>^{\frac{1}{2}}f\|_X,
	\end{align}for any sufficiently smooth $f\in D(\Lambda)=D(\L)$ with $D(\L)$ given in \eqref{def.ddl}, where $C_0$ (large) is a parameter in \eqref{X}. 
\end{Lem}

\begin{proof}
It follows from \eqref{W1} that   
\begin{align*}
	W\Lambda f = v\cdot \nabla_x (Wf) + \frac{q|v|^2}{\langle v\rangle}Wf + WAf.
\end{align*}
	Using the zero-inflow boundary condition \eqref{inflow} in \eqref{X1}, we have 
	\begin{align}\notag\label{eq2}
		&\quad\,(W\Lambda f,Wf)_{L^2_{x,v}} = \big(v\cdot \nabla_x (Wf) + \frac{q|v|^2}{\langle v\rangle}Wf + WAf,Wf\big)_{L^2_{x,v}} \\
		&\notag= \int_{\pa\Omega}\int_{v\cdot n>0}v\cdot n(x)|Wf|^2\,dvdS(x) +\int_{\pa\Omega}\int_{v\cdot n<0}v\cdot n(x)|Wg|^2\,dvdS(x)\\&\notag\qquad+q(|v|^2\<v\>^{-1}Wf,Wf)_{L^2_{x,v}} + (W^2Af,f)_{L^2_{x,v}}\\
		&\ge q\||v|\<v\>^{-\frac{1}{2}}Wf\|^2_{L^2_{x,v}} + (W^2Af,f)_{L^2_{x,v}},
	\end{align}
	and 
	\begin{align}\label{eq1}\notag
		\quad\,(\Lambda f,f)_{L^2_{x,v}} = (v\cdot\na_xf+A f,f)_{L^2_{x,v}} 
		&\notag\ge \int_{\pa\Omega}\int_{v\cdot n>0}v\cdot n(x)|f|^2\,dvdS(x)+(Af,f)_{L^2_{x,v}}
		\\&\ge (Af,f)_{L^2_{x,v}}\ge \lambda\|f\|_{L^2_xL^2_D}^2, 
	\end{align}
with some constant $\lambda>0$. 
Here the last inequality follows from \eqref{16}. 

	In view of \eqref{eq2} and \eqref{eq1}, we only need to find the upper bound of $(WAf,Wf)_{L^2_{x,v}}$. By \eqref{e22}, we know that 
	\begin{align*}
		|(W Af,Wf)_{L^2_{x,v}}|&\lesssim \|f\|_{L^2_xL^2_D}^2.
	\end{align*}
	Thus, choosing $C_0>0$ sufficiently large, we deduce from definition \eqref{X1}, \eqref{eq2} and \eqref{eq1} that 
	\begin{align*}
		(\Lambda f,f)_X&\ge C_0\lam\|f\|^2_{L^2_xL^2_D}+q\|\<v\>^{-\frac{1}{2}}|v|Wf\|^2_{L^2_{x,v}}-C\|f\|^2_{L^2_xL^2_D}\\
		&\ge \frac{C_0\lam}{2}\|f\|^2_{L^2_xL^2_D}+q\|\<v\>^{-\frac{1}{2}}|v|Wf\|^2_{L^2_{x,v}}\\
		&\ge c_0\|\<v\>^{1/2}f\|_X^2,
	\end{align*}
	for a small constant $c_0>0$. Note that $\|\nu f\|_{L^2_xL^2_v}\lesssim \|f\|_{L^2_xL^2_D}$ and $\<v\>^{\frac{1}{2}}\lesssim \nu + \<v\>^{-\frac{1}{2}}|v|$. 
%
%
	This then completes the proof of Lemma \ref{lem.coxv}.
\end{proof}

In the following, we denote $\Lambda=v\cdot\nabla_x-A$ is the sense of distribution and denote its domain $ D(\Lambda) = \{f\in X:f(t,x,v) = 0\ \text{ on } \gamma_-,\ \Lambda f\in X\}$. Then $\Lambda$ is closed, cf.~\cite{MKM}.
Indeed, for any sequence $f_n\in  D(\Lambda)$ such that $f_n\to f$ and $\Lambda f_n\to h$ for some $f,g\in X$, we have that for any test function $\varphi\in \mathcal{C}^\infty_c(\Omega\times\R^3_v)$, 
\begin{align*}
	(h,\varphi)_{L^2_{x,v}} =\lim_{n\to\infty}(\Lambda f_n,\varphi)_{L^2_{x,v}}
	=\lim_{n\to\infty}( f_n,\Lambda'\varphi)_{L^2_{x,v}}
	=( f,\Lambda'\varphi)_{L^2_{x,v}}=(\Lambda f,\varphi)_{L^2_{x,v}}.
\end{align*}
Thus $\Lambda f=g\in X$. Moreover, in term of weak formulation, it holds that $f=0$ on $\gamma_-$, and hence, $f\in  D(\Lambda)$. Then we can analyze the spectrum structure of operator $\Lambda$ on $D(\Lambda)$. The idea of proof for the following Theorem comes from \cite{Alonso2020}. 

\begin{Thm}\label{thm.spex}Let $q>0$ in \eqref{W}. For both Landau case and Boltzmann case, the spectrum of $(-\Lambda,D(\Lambda))$, as an operator on $(X,\|\cdot\|_X)$ defined in \eqref{X1} and \eqref{X} with the definite domain $D(\Lambda)=D(\L)$ given in \eqref{def.ddl}, lies in $\{z\in\C: \Re\, z\le -c_0\}$, 
that is 
	\begin{align*}
		\sigma(-\Lambda)\subset \{z\in\C: \Re\, z\le -c_0\}. 
	\end{align*}
\end{Thm}
\begin{proof}
	Firstly, $c_0-\Lambda$ is dissipative on $X$ by \eqref{10}. Secondly, the domain $ D(\Lambda)$ is dense in $X$, since it contains $\{f\in X:f\in \mathcal{C}_c^\infty(\Omega\times\R^3)\}$ functions. By Proposition II.3.14 in \cite{Engel1999}, it suffices to prove the existence and uniqueness of the problem 
	\begin{align*}
		(\lambda + \Lambda)f = g ,\quad f=0\text{ on }\gamma_-, 
	\end{align*}
for $g\in X$ and $\lambda\in\C$ with $\Re\,\lam>-c_0$. 
	Write $f=f_R+if_I$, $g=g_R+ig_I$ and $\lambda= \lambda_R+i\lambda_I$, then the above problem is equivalent to the real-valued problems:
	\begin{align}\label{21}
		\left(\big(\Lambda+\lambda_R\big)\1_{2\times2}+\lambda_I\begin{bmatrix}
			0 & -1\\1& 0
		\end{bmatrix}\right)\begin{bmatrix}
			f_R\\f_I
		\end{bmatrix}=\begin{bmatrix}
			g_R\\g_I
		\end{bmatrix},
	\end{align}
with given $g$. 
	To solve this problem, we consider its perturbation form:
	\begin{align}\label{22}
		\left(\big(\epsilon\mathcal{L}_P+\Lambda+\lambda_R\big)\1_{2\times2}+\lambda_I\begin{bmatrix}
			0 & -1\\1& 0
		\end{bmatrix}\right)\begin{bmatrix}
			f_R\\f_I
		\end{bmatrix}=\begin{bmatrix}
			g_R\\g_I
		\end{bmatrix},
	\end{align}
	where 
	\begin{align*}
		\mathcal{L}_P := -\nabla_{x,v}\cdot\<v\>^{\gamma+4}\nabla_{x,v}+\<v\>^{\gamma+4}.
	\end{align*}
In order to eliminate the boundary effect, we consider the following weak form to equation \eqref{22} on $X$. That is, for any 
	$\varphi=\begin{bmatrix}
		\varphi_R\\\varphi_I
	\end{bmatrix}$, 
	we write  
	\begin{align}	
		&(W\Lambda f,W\varphi)_{L^2_{x,v}}+C_0(\Lambda f,\varphi)_{L^2_{x,v}}
		\notag+\epsilon\Big((W^2\<v\>^{\gamma+4}\nabla_{x,v}f,\nabla_{x,v}\varphi)_{L^2_{x,v}}
		\\&\qquad+C_0(\<v\>^{\gamma+4}\nabla_{x,v}f,\nabla_{x,v}\varphi)_{L^2_{x,v}}+(W^2\<v\>^{\gamma+4}f,\varphi)_{L^2_{x,v}}+C_0(\<v\>^{\gamma+4}f,\varphi)_{L^2_{x,v}}\Big)\notag\\
		&\quad
		+\lambda_R(f,\varphi)_{X}+\lambda_I\Big(\begin{bmatrix}
			0 & -1\\1& 0
		\end{bmatrix}f,\varphi\Big)_{X} = (g,\varphi)_{X}.
		\label{23}
	\end{align}
	Then we introduce the bilinear form $B^\epsilon[f,\varphi]:\mathcal{H}\times\mathcal{H}\to \R$, which is defined by the left-hand side of \eqref{23} with 
	\begin{multline*}
		\mathcal{H} := \Big\{g=(g_R,g_I)\in\Big(H^1_{x,v}(\<v\>^{\frac{\gamma+4}{2}})\cap L^2_{x,v}(\<v\>^{\frac{\gamma+4}{2}})\Big)\times\Big(H^1_{x,v}(\<v\>^{\frac{\gamma+4}{2}})\cap L^2_{x,v}(\<v\>^{\frac{\gamma+4}{2}})\Big) :\\
		 g_R(t,x,v)=g_I(t,x,v)=0 \text{ on }\gamma_-\Big\}.
	\end{multline*}
Here  the weighted space $H^1_{x,v}(\<v\>^{\frac{\gamma+4}{2}})$ and $L^2_{x,v}(\<v\>^{\frac{\gamma+4}{2}})$ contain functions that are finite with respect to norms  
	\begin{align*}
		\sum_{|\alpha|+|\beta|\le 1}\|\<v\>^{\frac{\gamma+4}{2}}\partial^{\alpha}_\beta g\|_{L^2_{x,v}(\Omega\times \R^3)}\quad\text{ and }\quad\sum_{|\alpha|+|\beta|\le 1}\|\<v\>^{\frac{\gamma+4}{2}}g\|_{L^2_{x,v}(\Omega\times \R^3)}, 
	\end{align*}
	respectively. 
	Thanks to \eqref{10}, by choosing $C_0>0$ large enough, for $f\in \mathcal{H}$, we have 
	\begin{align}\label{213}\notag
		B^\epsilon[f,f]&\ge (c_0+\lambda_R)\Big(\frac{C_0}{2}\|f\|^2_{L^2_xL^2_{D}} + q\|\<v\>^{-\frac{1}{2}}|v|Wf\|^2_{L^2_{x,v}}\Big) \\
		&\qquad+ \epsilon\|\<v\>^{\frac{\gamma+4}{2}}\nabla_vf\|^2_{X\times X} + \epsilon\|\<v\>^{\frac{\gamma+4}{2}}\nabla_xf\|^2_{X\times X} + \epsilon\|\<v\>^{\frac{\gamma+4}{2}}f\|_{X\times X}^2. 
	\end{align}
	Note that  the anti-symmetric term related to $\lambda_I$ vanishes. 
%
	On the other hand, using \eqref{27a}, \eqref{27b} and \eqref{defnu}, we have  $|(Af,f)_{L^2_v}|\lesssim |f|_{L^2_D}^2\lesssim |\<v\>^{\frac{\gamma+2}{2}}\<D_v\>f|_{L^2_v}^2$. Then it's direct to check that 
	\begin{align*}
		|B^\epsilon[g,f]|\le (C+|\lambda|+\epsilon)\|g\|_{\mathcal{H}}\|f\|_{\mathcal{H}}.
	\end{align*}
	Therefore, whenever $c_0+\lambda_R>0$, by Lax-Milgram Theorem, for any $g$ in the dual of $\mathcal{H}$ (and hence for any $g\in X\times X$), there exists a unique $f^\epsilon\in \mathcal{H}$ such that 
	\begin{align*}
		B^\epsilon[f^\epsilon,\varphi] = (g,\varphi)_X, \quad \forall\, \varphi\in\mathcal{H},\ \forall\,\epsilon>0.
	\end{align*}
	This gives the existence and uniqueness for problem \eqref{23} whenever $c_0+\lambda_R>0$. 
	Using \eqref{213}, the weak solution $f^\epsilon$ satisfies 
	\begin{align*}
		 (c_0+\lambda_R)\Big(\frac{C_0}{2}\|f^\epsilon\|^2_{L^2_xL^2_{D}} + q\|\<v\>^{-\frac{1}{2}}|v|Wf^\epsilon\|^2_{L^2_{x,v}}\Big)\le B^\epsilon[f^\epsilon,f^\epsilon]= (g,f^\epsilon)_X\le \|g\|_{X\times X}\|f^\epsilon\|_{X\times X}.
	\end{align*}
	Thus, noticing $\|\cdot\|_X\approx \|\cdot\|_{L^2_{x,v}}\lesssim \|\cdot\|_{L^2_xL^2_{D}} + q\|\<v\>^{-\frac{1}{2}}|v|W(\cdot)\|_{L^2_{x,v}}$, we have 
	\begin{align*}
		  \frac{C_0}{2}\|f^\epsilon\|_{L^2_xL^2_{D}} + q\|\<v\>^{-\frac{1}{2}}|v|Wf^\epsilon\|_{L^2_{x,v}}\lesssim (c_0+\lambda_R)^{-1}\|g\|_{X\times X}.
	\end{align*}
	Thus for fixed $g\in X\times X$, $\|f^\epsilon\|_{L^2_xL^2_{D}}$ and $\|\<v\>^{-\frac{1}{2}}|v|Wf^\epsilon\|_{L^2_{x,v}}$ are bounded  as $\epsilon\to 0$. We let $f_g\in X\times X$ to be the weak limit of $\{f^\epsilon\}$ as $\epsilon\to 0$. Then $f_g$ satisfies problem \eqref{21} in the sense of distribution with estimate 
	\begin{align}\label{25}
	\|f_g\|_{X\times X}\lesssim \frac{C_0}{2}\|f_g\|_{L^2_xL^2_{D}} + q\|\<v\>^{-\frac{1}{2}}|v|Wf_g\|_{L^2_{x,v}}\le (c_0+\lambda_R)^{-1}\|g\|_{X\times X}.
	\end{align}  
	Furthermore, any solution to \eqref{21} in $X\times X$ satisfies estimate \eqref{25}. Thus, the solution is unique in this space. Note that we define $\Lambda$ in the sense of distribution.  
	This proves that $\lambda+\Lambda: D(\Lambda)\to X$ is bijective whenever $\lam_R>-c_0$. 
	By Proposition II.3.14 in \cite{Engel1999}, we know that any $\lambda\in\C$ such that $\lambda_R>-c_0$ belongs to the spectrum of $\Lambda$ in space $X$.  This completes the proof of Theorem \ref{thm.spex}.
\end{proof}


In order to obtain the spatial regularity effect of operator $\Lambda$, we begin with the following extension theorem. Once we could extend the function in $\Omega$ with boundary values to a function in $\R^3_x$, then we can easily apply the technique in the whole space such as the Fourier transform to deduce the regularizing effect of $\Lambda$.

\begin{Lem}\label{Lem23}
	Let $g\in L^2(\Omega\times\R^3_v)$ such that $\Lambda g\in L^2(\Omega\times\R^3_v)$ and $g=0$ on $\gamma_-$. Then there exists an extension operator $E$ such that 
	\begin{align*}
		Eg|_{\Omega\times\R^3_v}=g,\qquad  Eg|_{\pa\Omega\times\R^3_v}=g|_{\pa\Omega\times\R^3_v},
	\end{align*}
and 
\begin{equation*}
	\Lambda Eg = \left\{\begin{aligned}
		&\Lambda g,\quad\text{ on }\Omega\times\R^3_v,\\
		&0,\qquad\text{ on }\Omega^c\times\R^3_v.
	\end{aligned}\right.
\end{equation*}
Moreover, 
\begin{align*}
	\|Eg\|_{L^2(\R^3\times\R^3_v)}+\|\Lambda Eg\|_{L^2(\R^3\times\R^3_v)}\le \|g\|_{L^2(\Omega\times\R^3_v)}+\|\Lambda g\|_{L^2(\Omega\times\R^3_v)}. 
\end{align*}
\end{Lem}

\begin{proof}
	We proceed in three steps. 
	
	\noindent{\bf Step 1.}
	We firstly find a smoothing version of $g$ by considering the problem in $\Omega$:
	\begin{align}\label{211}
		B^\ve_1[f,\varphi]=(g+\Lambda g,\varphi)_{L^2(\Omega\times\R^3)},\quad f=0\text{ on }\gamma_-,
	\end{align} for $\varphi\in \mathcal{H}_1$ and $\ve>0$.
	Here  $B^\ve_1:\mathcal{H}_1\to \R$ and Hilbert space $\H_1$ are given respectively by 
	\begin{align*}
		B^\ve_1[f,\varphi] := &(f+\Lambda f,\varphi)_{L^2_{x,v}(\Omega\times\R^3)}\\
		\notag&+\ve\Big((\<v\>^{\gamma+4}\nabla_{x,v}f,\nabla_{x,v}\varphi)_{L^2_{x,v}(\Omega\times\R^3)}+(\<v\>^{\gamma+4}f,\varphi)_{L^2_{x,v}(\Omega\times\R^3)}\Big),
	\end{align*}
	and 
	\begin{equation*}
		\mathcal{H}_1 := \Big\{f\in\Big(H^1_{x,v}(\Omega\times\R^3_v)(\<v\>^{\frac{\gamma+4}{2}})\cap L^2_{x,v}(\Omega\times\R^3_v)(\<v\>^{\frac{\gamma+4}{2}})\Big):f=0\text{ on }\gamma_-\Big\}.
	\end{equation*}
	Here $H^1_{x,v}(\Omega\times\R^3_v)(\<v\>^{\frac{\gamma+4}{2}})$ and $L^2_{x,v}(\Omega\times\R^3_v)(\<v\>^{\frac{\gamma+4}{2}})$ contains functions that are finite with respect to the following norms 
	\begin{align*}
		\sum_{|\alpha|+|\beta|\le 1}\|\<v\>^{\frac{\gamma+4}{2}}\partial^{\alpha}_\beta g\|_{L^2_{x,v}(\Omega\times \R^3)}\quad\text{ and }\quad \|\<v\>^{\frac{\gamma+4}{2}} g\|_{L^2_{x,v}(\Omega\times \R^3)}, 
	\end{align*}respectively.
Note that since $f\in\H_1$ implies $f=0$ on the boundary $\gamma_-$, we know that $\H_1$ is a Hilbert space, which is a necessary condition for Lax-Milgram Theorem. 
	Thanks to \eqref{eq1}, we have 
	\begin{align}\label{217}\notag
		B^\ve_1[f,f]&\ge \|f\|^2_{L^2(\Omega\times\R^3)}+\int_{\pa\Omega,\,v\cdot n>0}v\cdot n|f|^2\,dS(x)+c_0\|f\|^2_{L^2_x(\Omega)L^2_{D}}\\ &\quad+ \ve\|\<v\>^{\frac{\gamma+4}{2}}\nabla_{x,v}f\|^2_{L^2_{x,v}(\Omega\times \R^3)}+ \ve\|\<v\>^{\frac{\gamma+4}{2}}f\|_{L^2_{x,v}(\Omega\times \R^3)}^2. 
	\end{align}
	On the other hand, it's direct to check that 
	\begin{align*}
		|B^\ve_1[f,\varphi]|\le (C+\ve)\|f\|_{\mathcal{H}_1}\|\varphi\|_{\mathcal{H}_1}.
	\end{align*}
	Therefore, by Lax-Milgram Theorem, for any $(I+\Lambda) g$ in the dual of $\mathcal{H}_1$ (and hence for $(I+\Lambda) g\in L^2(\Omega\times\R^3)$), there exists a unique solution $f^\ve_1\in \mathcal{H}_1$ to problem \eqref{211} satisfying
	\begin{align}\label{217c}
		B^\ve_1[f^\ve_1,\varphi]= (g+\Lambda g,\varphi)_{L^2_{x,v}(\Omega\times\R^3)}, \quad f^\ve_1=0\text{ on }\gamma_-, 
	\end{align}
for any $\varphi\in\H_1$
and,
	\begin{align*}
		B^\ve_1[f^\ve_1,f^\ve_1]= (g+\Lambda g,f^\ve_1)_{L^2_{x,v}(\Omega\times\R^3)}
		\le \|(I+\Lambda) g\|_{L^2(\Omega\times\R^3)}\|f^\ve_1\|_{L^2(\Omega\times\R^3)}. 
	\end{align*}
	Together with \eqref{217}, we have 
	\begin{multline}\label{218}
		\|f^\ve_1\|^2_{L^2(\Omega\times\R^3)}+\int_{\pa\Omega,\,v\cdot n>0}v\cdot n|f^\ve_1|^2\,dS(x)+c_0\|f^\ve_1\|^2_{L^2_x(\Omega)L^2_{D}} \\+ \ve\|\<v\>^{\frac{\gamma+4}{2}}\nabla_{x,v}f^\ve_1\|^2_{L^2_{x,v}(\Omega\times \R^3)}+ \ve\|\<v\>^{\frac{\gamma+4}{2}}f^\ve_1\|_{L^2_{x,v}(\Omega\times \R^3)}^2\lesssim \|(I+\Lambda) g\|^2_{L^2_{x,v}(\Omega\times \R^3)}.
	\end{multline}
Then $\{f^\ve_1\}$ is bounded in $L^2(\Omega\times\R^3)$ and $L^2_x(\Omega)L^2_D$ and hence has a weak limit $f_1$  as $\ve\to0$ up to a subsequence. Taking limit $\ve\to 0$ in \eqref{217c} and applying \eqref{218}, we have 
\begin{align}\label{217b}
	f_1 + \Lambda f_1 = g+\Lambda g\quad \text{ in }\Omega\times\R^3,\qquad f_1=0\text{ on }\gamma_-.
\end{align}
Taking inner product of \eqref{217b} with $f_1-g$ over $\Omega\times\R^3$, we have 
\begin{align}\label{228}
	\|f_1-g\|^2_{L^2_{x,v}(\Omega\times \R^3)} +\int_{\pa\Omega}\int_{v\cdot n>0}v\cdot n|f_1-g|^2\,dvdS(x) + \|f_1-g\|^2_{L^2_x(\Omega)L^2_{D}} = 0.
\end{align}
This shows that the weak limit $f^\ve_1$ in $\Omega\times\R^3$ is equal to $g$ on $\ol\Omega\times\R^3$ and $f^\ve_1$ is the smooth approximated version of $g$. 
Next, we extend this smoothing approximated function to the whole space. 
	Using extension on Sobolev space for Lipschitz boundary domain (see for instance \cite[Theorem VI.5, pp. 181]{Stein1971}),
	there exists a linear operator $E$ mapping functions in $\Omega\times\R^3$ to functions in $\R^3\times\R^3$ such that $E(f^\ve_1)|_{\Omega\times\R^3}=f^\ve_1$ and 
	\begin{align*}
		\|E(f^\ve_1)\|_{H^1_x(\R^3)}\lesssim \|f^\ve_1\|_{H^1_x(\Omega)}.
	\end{align*}
	In view of this and \eqref{218}, we have 
	\begin{multline}\label{219}
		\|E(f^\ve_1)\|^2_{L^2(\R^3\times\R^3)}+c_0\|E(f^\ve_1)\|^2_{L^2_x(\R^3)L^2_{D}} + \ve\|\<v\>^{\frac{\gamma+4}{2}}\nabla_{x,v}E(f^\ve_1)\|^2_{L^2_{x,v}(\R^3\times \R^3)}\\+ \ve\|\<v\>^{\frac{\gamma+4}{2}}E(f^\ve_1)\|_{L^2_{x,v}(\R^3\times \R^3)}^2\lesssim \|(I+\Lambda) g\|^2_{L^2_{x,v}(\Omega\times\R^3)}.
	\end{multline}
This gives the boundedness of $E(f^\ve_1)$ in the whole space. 

	\medskip
\noindent{\bf Step 2.}
	In the following we extend function $f^\ve_1$ to $\Omega^c$ with respect to operator $\Lambda$. That is, we will solve the equation in $\Omega^c\times\R^3$:
	\begin{align}\label{211a}
		B^\eta_2[f,\varphi]=-(E(f^\ve_1)+\Lambda E(f^\ve_1),\varphi)_{L^2_{x,v}(\Omega^c\times\R^3)},\quad f=0\text{ on }\pa\Omega^c,
	\end{align} for $\eta>0$ and $\varphi\in \mathcal{H}_2$.
	Here $B^\eta_2:\mathcal{H}_2\to \R$ and Hilbert space $\H_2$ are given respectively by 
	\begin{align*}
		B^\eta_2[f,\varphi] :&= (f+\Lambda f,\varphi)_{L^2_{x,v}(\Omega^c\times\R^3)}\\
		\notag&\quad+\eta\Big((\<v\>^{\gamma+4}\nabla_{x,v}f,\nabla_{x,v}\varphi)_{L^2_{x,v}(\Omega^c\times\R^3)}+(\<v\>^{\gamma+4}f,\varphi)_{L^2_{x,v}(\Omega^c\times\R^3)}\Big),
	\end{align*}
	and 
	\begin{equation*}
		\mathcal{H}_2 := \Big\{f\in\Big(H^1_{x,v}(\Omega^c\times\R^3_v)(\<v\>^{\frac{\gamma+4}{2}})\cap L^2_{x,v}(\Omega^c\times\R^3_v)(\<v\>^{\frac{\gamma+4}{2}})\Big):f=0\text{ on }\pa\Omega^c\Big\}.
	\end{equation*}
	Since $f\in\mathcal{H}_2$ vanishes on boundary $\pa\Omega^c$, it's direct to check that 
	\begin{align}\label{222}\notag
		B^\eta_2[f,f]&\ge \|f\|^2_{L^2(\Omega^c\times\R^3)}+c_0\|f\|^2_{L^2_x(\Omega^c)L^2_{D}} \\&\quad+ \eta\|\<v\>^{\frac{\gamma+4}{2}}\nabla_{x,v}f\|^2_{L^2_{x,v}(\Omega^c\times \R^3)}+ \eta\|\<v\>^{\frac{\gamma+4}{2}}f\|_{L^2_{x,v}(\Omega^c\times \R^3)}^2, 
	\end{align}
	and $B^\eta_2[f,f]\le (C+\eta)\|f\|_{\mathcal{H}_2}\|g\|_{\mathcal{H}_2}$. 
	In view of \eqref{e22} and \eqref{219}, 
	we deduce from direct calculation on $\Lambda=v\cdot\na_x + A$ that  
	\begin{align*}
		|(E(f^\ve_1)+\Lambda E(f^\ve_1),\varphi)|
		&\lesssim \big(\|E(f^\ve_1)\|_{L^2_{x,v}(\R^3\times\R^3)}+\|\<v\>\na_xE(f^\ve_1)\|_{L^2_{x,v}(\R^3\times\R^3)}+\|E(f^\ve_1)\|_{L^2_x(\R^3)L^2_{D}}\big)\\
		&\qquad\times\big(\|\varphi\|_{L^2_{x,v}(\Omega^c\times\R^3)} +\|\varphi\|_{L^2_x(\Omega^c)L^2_D} \big)\\
		&\lesssim C_\ve\|(I+\Lambda) g\|_{L^2_{x,v}(\Omega\times\R^3)}\|\varphi\|_{\mathcal{H}_2},
	\end{align*}
for $\varphi\in\mathcal{H}_2$. 
	Therefore, by Lax-Milgram Theorem, there exists a unique solution $f^\eta_2\in\mathcal{H}_2$ to problem \eqref{211a}
	 satisfying 
	 \begin{align}\label{211c}
	 	B^\eta_2[f^\eta_2,\varphi]&=(E(f^\ve_1)+\Lambda E(f^\ve_1),\varphi),\quad f^\eta_2=0\text{ on }\pa\Omega^c,
	 \end{align}for $\varphi\in\mathcal{H}_2$
 and,
	\begin{align*}
		B^\eta_2[f^\eta_2,f^\eta_2]&=(E(f^\ve_1)+\Lambda E(f^\ve_1),f^\eta_2)\\
		&\le C_\ve\|(I+\Lambda) g\|_{L^2_{x,v}(\Omega\times\R^3)}\big(\|f^\eta_2\|_{L^2_{x,v}(\Omega^c\times\R^3)} +\|f^\eta_2\|_{L^2_x(\Omega^c)L^2_D} \big). 
	\end{align*}
Together with \eqref{222}, we have 
\begin{multline*}
	\|f^\eta_2\|^2_{L^2(\Omega^c\times\R^3)}+c_0\|f^\eta_2\|^2_{L^2_x(\Omega^c)L^2_{D}} + \eta\|\<v\>^{\frac{\gamma+4}{2}}\nabla_{x,v}f^\eta_2\|^2_{L^2_{x,v}(\Omega^c\times \R^3)}\\+ \eta\|\<v\>^{\frac{\gamma+4}{2}}f^\eta_2\|_{L^2_{x,v}(\Omega^c\times \R^3)}^2\le C_\ve\|(I+\Lambda) g\|_{L^2_{x,v}(\Omega\times\R^3)}^2. 
\end{multline*}
Then $\{f^\eta_2\}$ is bounded in $L^2(\Omega\times\R^3)$ and $L^2_x(\Omega)L^2_D$ and hence has a weak limit $f_2$  as $\eta\to0$ up to a subsequence. Taking limit $\eta\to 0$ in \eqref{211c}, we have 
\begin{align}\label{223}
(f_2+\Lambda f_2,\varphi)_{L^2_{x,v}(\Omega^c\times\R^3)}=-(E(f^\ve_1)+\Lambda E(f^\ve_1),\varphi)_{L^2_{x,v}(\Omega^c\times\R^3)},\quad f_2=0\text{ on }\pa\Omega^c. 
\end{align}
Denote $f^\ve_3=f_2+E(f^\ve_1)$ in $\Omega^c\times\R^3$. Then \eqref{223} implies that 
\begin{align}\label{225}
	f^\ve_3+\Lambda f^\ve_3=0 \text{ on }\Omega^c\times\R^3,\quad f^\ve_3=f^\ve_1\text{ on }\pa\Omega^c. 
\end{align}
Taking the inner product with $f^\ve_3$, we deduce from \eqref{225} that
\begin{align}\label{214}\notag
	\|f^\ve_3\|^2_{L^2(\Omega^c\times\R^3)}+\|f^\ve_3\|^2_{L^2(\Omega^c)L^2_D}&\le(f^\ve_3+\Lambda f^\ve_3,f^\ve_3)_{L^2(\Omega^c\times\R^3)} + \int_{\pa\Omega}v\cdot n|f^\ve_1|^2\,dS(x)\\
	&\le C\|g+\Lambda g\|^2_{L^2(\Omega\times\R^3)},
\end{align}
where the boundary term is controlled by using \eqref{218}.
Then $f^\ve_3$ is bounded in $L^2(\Omega^c\times\R^3)$ and $L^2(\Omega^c)L^2_D$. There exists a weak limit $f_3$ of $\{f^\ve_3\}$ as $\ve\to0$ up to a subsequence. Taking weak limit $\ve\to 0$ in \eqref{225}, we have 
\begin{align}\label{226}
	f_3+\Lambda f_3=0 \text{ on }\Omega^c\times\R^3, 
\end{align}
and hence
\begin{align}\label{214a}
	\|f_3\|^2_{L^2(\Omega^c\times\R^3)}+\|f_3\|^2_{L^2(\Omega^c)L^2_D}
	&\le C\|g+\Lambda g\|^2_{L^2(\Omega\times\R^3)}.
\end{align}
This function $f_3$ gives the extension of $g$ on $\Omega^c\times\R^3$. 

\medskip
\noindent{\bf Step 3.}
Now we denote 
\begin{equation*}
	f^\ve = \left\{\begin{aligned}
		&f^\ve_1,\text{ on } \Omega\times\R^3,\\
		&f^\ve_3,\text{ on } \Omega^c\times\R^3.
	\end{aligned}\right.
\end{equation*}
By \eqref{218} and \eqref{214}, we know that $f^\ve$ is uniformly bounded in $L^2(\R^3\times\R^3)$ and $L^2(\R^3)L^2_D$ with respect to $\ve$. Denote $f$ to be the weak limit of $\{f^\ve\}$ (up to subsequence). Since the weak limit is unique, we have from \eqref{217b} and \eqref{226} that 
\begin{equation*}
	f = \left\{\begin{aligned}
		&f_1,\text{ on } \Omega\times\R^3,\\
		&f_3,\text{ on } \Omega^c\times\R^3,
	\end{aligned}\right.
\end{equation*}
and 
\begin{equation*}
	f+\Lambda f = \left\{\begin{aligned}
		&g+\Lambda g,\text{ on } \Omega\times\R^3,\\
		&0,\qquad\text{ on } \Omega^c\times\R^3. 
	\end{aligned}\right.
\end{equation*}
Note that when we take the weak limit above, the boundary term arsing from operator $v\cdot\na_x$  along $\pa\Omega$ vanishes since $f^\ve_1$ in $\Omega$ and $f^\ve_3$ in $\Omega^c$ have the same boundary values and the normal outward vectors on $\Omega$ and $\Omega^c$ have different signs. In fact, for any smooth test function $\varphi$, we have 
\begin{align*}
	&\quad\,\lim_{\ve\to 0}(f^\ve+\Lambda f^\ve,\varphi)_{L^2(\R^3\times\R^3)}\\ &= \lim_{\ve\to 0}\Big\{(f^\ve_1+v\cdot\na_x f^\ve_1+Af^\ve_1,\varphi)_{L^2(\Omega\times\R^3)} + (f^\ve_3+v\cdot\na_x f^\ve_3+Af^\ve_3,\varphi)_{L^2(\Omega^c\times\R^3)}\Big\}\\
	&= \lim_{\ve\to 0}\Big\{(f^\ve_1,\varphi-v\cdot\na_x\varphi+A^*\varphi)_{L^2(\Omega\times\R^3)} + \int_{\pa\Omega}\int_{\R^3}v\cdot n(x)f^\ve_1\varphi\,dvdS(x) \\&\qquad\qquad+ (f^\ve_3,\varphi-v\cdot\na_x\varphi+A^*\varphi)_{L^2(\Omega^c\times\R^3)}+\int_{\pa\Omega^c}\int_{\R^3}v\cdot n(x)f^\ve_1\varphi\,dvdS(x)\Big\}\\
	&= (f_1,\varphi-v\cdot\na_x\varphi+A^*\varphi)_{L^2(\Omega\times\R^3)}+(f_3,\varphi-v\cdot\na_x\varphi+A^*\varphi)_{L^2(\Omega^c\times\R^3)}\\& = (g+\Lambda g,\varphi)_{L^2(\Omega\times\R^3)}. 
\end{align*}
Write $Eg:=f$ and we conclude Lemma \ref{Lem23} from \eqref{228}, \eqref{226} and \eqref{214a}. 
\end{proof}

Applying the above extension theorem, we can now prove Theorem \ref{Main1}. 
\begin{proof}[Proof of Theorem \ref{Main1}]
	Let $R>0$ be given in Lemma \ref{Lem21}. We proceed it in two steps as follows. Note that we only consider Landau case and non-cutoff Boltzmann case in this theorem. 
	
%


	\medskip
	\noindent{\bf{Step 1.}} In this step we prove that $K$ is $\Lambda$-compact in $X$. Taking a sequence $\{g_n\}\subset D(\Lambda)\subset X$ such that both $\{g_n\}$ and $\{\Lambda g_n\}$ are bounded in $X$, we apply Lemma \ref{Lem23} to extend function $g_n$ to $Eg_n$ over $\R^3\times\R^3$ such that 
	\begin{align}\label{28}
		\|Eg_n\|_{L^2(\R^3\times\R^3_v)}+\|\Lambda Eg_n\|_{L^2(\R^3\times\R^3)}\lesssim\|g_n\|_{L^2(\Omega\times\R^3_v)}+ \|\Lambda g_n\|_{L^2(\Omega\times\R^3)}. 
	\end{align}
	Then it follows from  \eqref{16} that 
	\begin{align}\label{29}
		\|Eg_n\|^2_{L^2_x(\R^3)L^2_D}\lesssim (\Lambda g_n,g_n)_{L^2(\R^3\times\R^3)}\le \|\Lambda Eg_n\|_{L^2(\R^3\times\R^3)}\|Eg_n\|_{L^2(\R^3\times\R^3)}.
	\end{align}
By definition \eqref{18a} and \eqref{18b} for $K$, we know that $KEg_n$ contains the smooth cut-off function $\chi_R$, where $\chi_R$ is given by \eqref{def.chieps}.
Since $\chi_R(v)$ is smooth and has compact support, $\<D_v\>^s\chi_R(v)$ can be regarded as a pseudo-differential operator with symbol in $S(\<v\>^{\frac{\gamma}{2}}\<\eta\>^{s})$, where we let $s=1$ for Landau case. By \cite[Lemma 2.4 and Corollary 2.5]{Deng2020a}, we know that 
\begin{multline}\label{222b}
	\|\<D_v\>^{s}KEg_n\|_{L^2(\R^3\times\R^3)}\lesssim \|\<v\>^{\frac{\gamma}{2}}\<D_v\>^sEg_n\|_{L^2(\R^3\times\R^3)}\\\lesssim\|Eg_n\|_{L^2_x(\R^3)L^2_D}\lesssim \|g_n\|_{L^2(\Omega\times\R^3_v)}+ \|\Lambda g_n\|_{L^2(\Omega\times\R^3)}, 
\end{multline}
where the second inequality follows from \eqref{27a} and \eqref{27b} and the last inequality follows from \eqref{28} and \eqref{29}. This yields the regularizing effect of $\Lambda$ with respect to velocity variable. 


	\smallskip

	Next we will derive the regularity on $x$. Temporarily we denote $f=Eg_n$ and denote Fourier transform $\F f=\int_{\R^3_x}fe^{-ix\cdot k}\,dx$ with respect to $x$. 
It follows from \eqref{17} that 
\begin{align}\label{201}
	|K\wh f|_{L^2_v(\R^3)}\lesssim |\chi_R \wh f|_{L^2_v(\R^3)}. 
\end{align}
Since $f$ is defined on $\R^3_x\times\R^3_v$, we can take Fourier transform to obtain 
\begin{align*}
	\lambda {\chi_R\wh f}+iv\cdot k{\chi_R\wh f} +{\chi_RA\wh f} = \lambda  {\chi_R\wh f}+{\chi_R\Lambda\wh f},
\end{align*}
for any $\lambda >1$. Then 
\begin{align}\label{220a}
	 {\chi_R\wh f}(k,v) = \frac{\lambda  {\chi_R\wh f}+{\chi_R\Lambda\wh f} - {\chi_RA \wh f}}{\lambda +iv\cdot k}.
\end{align}
We introduce a smooth mollifier in velocity for any $\ve\in(0,1)$:
\begin{align*}
	\rho_\ve(v) = \ve^{-3}\rho(\ve^{-1}v),\quad \rho\in C^\infty_c(|v|\le 1;[0,1]),\quad \int_{\R^3}\rho\,dv=1.
\end{align*}
Then 
\begin{align}\label{220}
	\chi_R\wh f = \big(\chi_R\wh f-\rho_\ve*_v\chi_R\wh f\big) + \rho_\ve*_v(\chi_R\wh f),
\end{align}
where $*_v$ is the convolution with respect to $v$. We can estimate the first right-hand term of \eqref{220} as 
\begin{align*}
	|\chi_R\wh f-\rho_\ve*_v\chi_R\wh f|_{L^2_{v}}
	&\le \int_{\R^3}|\big(\chi_R\wh f(v)-\chi_R\wh f(v-u)\big)\rho_\ve(u)|_{L^2_v}\,du\\
	&\le \Big(\int_{\R^3\times\R^3}\frac{|\chi_R\wh f(v)-\chi_R\wh f(v-u)|^2}{|u|^{3+2s}}\,dudv\Big)^{\frac{1}{2}}\Big(\int_{\R^3}|\rho_\ve(u)|^2|u|^{3+2s}\,du\Big)^{\frac{1}{2}}\\
	&\le C\ve^s|\<D_v\>^s\chi_R\wh f|_{L^2_v(\R^3)}.
\end{align*}
Similar to \eqref{222b}, we can regard $\<D_v\>^s\chi_R(v)$ as a pseudo-differential operator with symbol in $S(\<v\>^{\frac{\gamma}{2}}\<\eta\>^{s})$. By \cite[Lemma 2.4 and Corollary 2.5]{Deng2020a}, we know that 
\begin{align}\notag\label{221}
	|\chi_R\wh f-\rho_\ve*_v\chi_R\wh f|_{L^2_{v}}&\le C\ve^s|\<D_v\>^s\chi_R\wh f|_{L^2_v(\R^3)}\\
	&\le C\ve^s|\<v\>^{\frac{\gamma}{2}}\<D_v\>^{s}\wh{f}|_{L^2_v}\le C\ve^s|\wh{f}|_{L^2_D}. 
\end{align}
Here the last inequality follows from \eqref{27a} and \eqref{27b}. 
For the second part in \eqref{220}, it follows from \eqref{220a} that 
\begin{align*}
	|\rho_\ve*_v(\chi_R\wh f)| = \Big|\int_{\R^3}\frac{\lambda \chi_R\wh {f}(u)+\chi_R{\Lambda \wh f}(u) - \chi_R{A\wh f}(u)}{\lambda +iu\cdot k}\rho_\ve(v-u)\,du\Big|
	\le \sum_{i=1}^2I_i,
\end{align*}
where 
\begin{align*}
	I_1 &= \Big|\int_{\R^3}\frac{\lambda \chi_R(u)\wh {f}(u)+\chi_R(u){\Lambda\wh f}(u)}{\lambda +iu\cdot k}\rho_\ve(v-u)\,du\Big|,\\
	I_2 &= \Big|\int_{\R^3}\frac{-\chi_R(u){A\wh f}(u)}{\lambda +iu\cdot k}\rho_\ve(v-u)\,du\Big|.
\end{align*}
Then by H\"{o}lder's inequality, 
\begin{align}\label{216}
	I_1\le \Big(\int_{\R^3}\chi_R(u)|\lambda \wh {f}(u)+\wh{\Lambda f}(u)|^2|\rho_\ve(v-u)|\,du\Big)^{1/2}\Big(\int_{\R^3}\frac{\chi_R(u)|\rho_\ve(v-u)|}{\lambda^2 +|u\cdot k|^2}\,du\Big)^{1/2}.
\end{align}
For the second factor, we apply the method in \cite[Proposition 1.1]{Bouchut2002}. Write $u=\wt u\frac{k}{|k|}+u'$ with $\wt u=\frac{u\cdot k}{|k|}$ and $u'=u-\wt u\frac{k}{|k|}$. Then $k\perp u'$ and 
\begin{align}\label{216a}
	\int_{\R^3}\frac{\chi_R(u)|\rho_\ve(v-u)|}{\lambda^2 +|u\cdot k|^2}\,du
	&\le \frac{C}{\lambda^2}\int_{\R}\frac{\frac{1}{\ve}\1_{|\wt u|\le 2R}\1_{|\frac{v\cdot k}{|k|}-\wt u|<\ve}}{1+|\wt u|^2 |k|^2/\lambda^2}\,d\wt u
	\le \frac{C}{\lambda|k|}. 
\end{align}
Here $\1_{\{\cdot\}}$ is the indicator function on a set $\{\cdot\}$. For the term $I_2$, we have from \eqref{16a}, \eqref{27a} and \eqref{27b} that 
\begin{align}\label{216c}
	I_2\le C|\wh{f}|_{L^2_D}\Big|\frac{-\chi_R(u)}{\lambda +iu\cdot k}\rho_\ve(v-u)\Big|_{L^2_D(\R^3_u)}
	&\le C|\wh{f}|_{L^2_D}\Big|\frac{\<u\>^{\frac{\gamma}{2}}\chi_R(u)}{\lambda +iu\cdot k}\rho_\ve(v-u)\Big|_{H^1(\R^3_u)}.
\end{align}
Here we apply $|\<v\>^{\frac{\gamma}{2}}\<D_v\>^s(\cdot)|_{L^2_v}\approx |\<D_v\>^s\<v\>^{\frac{\gamma}{2}}(\cdot)|_{L^2_v}\lesssim |\<D_v\>\<v\>^{\frac{\gamma}{2}}(\cdot)|_{L^2_v}$ (see for instance \cite[Corollary 2.5]{Deng2020a}).
For the second factor in \eqref{216c}, it's direct to obtain that 
\begin{multline}\label{216d}
	\Big|\frac{\<u\>^{\frac{\gamma}{2}}\chi_R(u)}{\lambda +iu\cdot k}\rho_\ve(v-u)\Big|_{H^1(\R^3_u)}^2
	\le C_R\Big(\int_{\R^3}\frac{|\chi_R(u)\rho_\ve(v-u)|^2}{\lambda^2 +|u\cdot k|^2}\,du
	+\int_{\R^3}\frac{|\na_u\chi_R(u)\rho_\ve(v-u)|^2}{\lambda^2 +|u\cdot k|^2}\,du\\
	\qquad+\ve^{-2}\int_{\R^3}\frac{|\chi_R(u)\na_u\rho_\ve(v-u)|^2}{\lambda^2 +|u\cdot k|^2}\,du
	+|k|^2\int_{\R^3}\frac{|\chi_R(u)\rho_\ve(v-u)|^2}{(\lambda^2 +|u\cdot k|^2)^2}\,du\Big). 
\end{multline}
As in \cite{Ukai}, for any $\rho>0$, we set 
\begin{align*}
	\Sigma_1 &= \{u\in\R^3:|u|\le 2R,\ |u\cdot k|\le \rho|k|\},\\
	\Sigma_2 &= \{u\in\R^3:|u|\le 2R\} - \Sigma_1.
\end{align*}
Then their Lebesgue measures satisfy 
\begin{align*}
	|\Sigma_1|\le C\rho R^2, \quad |\Sigma_2|\le C R^3. 
\end{align*}
Splitting the integration region into $\Sigma_1$ and $\Sigma_2$ and choosing $\rho = \lambda^{\frac{2}{3}}|k|^{-\frac{2}{3}}$, we have 
\begin{align*}
	\int_{|u|\le 2R}\frac{1}{\lambda^2 +|u\cdot k|^2}\,du\le C_R\big(\lambda^{-2}\rho+(\rho|k|)^{-2}\big)\le C_R\lambda^{-\frac{4}{3}}|k|^{-\frac{2}{3}}.
\end{align*}
Similarly, choosing $\rho=\lambda^{4/5}|k|^{-4/5}$, we have 
\begin{align*}
	\int_{|u|\le 2R}\frac{1}{(\lambda^2 +|u\cdot k|^2)^2}\,du\le C_R(\lambda^{-4}\rho + (\rho|k|)^4)\le C_R\lambda^{-\frac{16}{5}}|k|^{-4/5}. 
\end{align*}
Combining the above two estimates, we take integration of \eqref{216d} with respect to $v$ to obtain 
\begin{align}\label{216e}
	\Big|\frac{\<u\>^{\frac{\gamma}{2}}\chi_R(u)}{\lambda +iu\cdot k}\rho_\ve(v-u)\Big|_{H^1(\R^3_u)L^2_v}^2
	&\le C_R\big((\ve^{-3}+\ve^{-5})\lambda^{-\frac{4}{3}}|k|^{-\frac{2}{3}} + \ve^{-3}\lambda^{-\frac{16}{5}}|k|^{\frac{6}{5}}\big).
\end{align}
Collecting \eqref{216}, \eqref{216a}, \eqref{216c} and \eqref{216e}, we have 
\begin{align}\label{221a}
	|\rho_\ve*_v(\chi_R\wh f)|_{L^2_v}&\le 
	\frac{C\lambda^{\frac{1}{2}}}{|k|^{\frac{1}{2}}}| \wh{f}(u)|_{L^2_u}+\frac{C}{\lambda^{\frac{1}{2}}|k|^{\frac{1}{2}}}|\wh{\Lambda f}(u)|_{L^2_u}+
	C_R\big(\ve^{-5}\lambda^{-\frac{4}{3}}|k|^{-\frac{2}{3}} + \ve^{-3}\lambda^{-\frac{16}{5}}|k|^{\frac{6}{5}}\big)|\wh f|_{L^2_D}. 
\end{align}
Now we assume $|k|\ge 1$ and choose $\ve=|k|^{-\frac{56}{195}}$, $\lambda=|k|^{\frac{9}{13}}$. Then \eqref{221} and \eqref{221a} yield 
\begin{align}\label{221b}
	|\chi_R\wh f(k)|_{L^2_v}\le C_R|k|^{-\frac{56}{195}s}|\wh f|_{L^2_D} + C_R|k|^{-\frac{2}{13}}\big(|\wh{f}|_{L^2_v}+|\wh{\Lambda f}|_{L^2_v}+|\wh f|_{L^2_D}\big).
\end{align}
Let $s_1=\min\{\frac{56}{195}s,\frac{2}{13}\}$. Applying \eqref{221b} for the parts $|k|>1$ and using Plancherel's Theorem, one can obtain 
\begin{align*}
	\Big(\int_{\R^3}\<k\>^{2s_1}|\chi_R\wh f(k)|^2_{L^2_v}\,dk\Big)^{\frac{1}{2}}
	&\lesssim \Big(\int_{\R^3}\<k\>^{s_1}\1_{|k|\le 1}|\chi_R\wh f(k)|_{L^2_v}\,dk\Big)^{\frac{1}{2}}+\Big(\int_{\R^3}\<k\>^{s_1}\1_{|k|> 1}|\chi_R\wh f(k)|_{L^2_v}\,dk\Big)^{\frac{1}{2}}\\
	&\lesssim \|{f}\|_{L^2_xL^2_v}+\|{\Lambda f}\|_{L^2_xL^2_v}+\|f\|_{L^2_xL^2_D}. 
\end{align*}
Recall that $f=Eg_n$. Together with \eqref{201}, \eqref{28} and \eqref{29}, we have  
\begin{align}\label{222a}
	\|\<D_x\>^{s_1}KEg_n\|_{L^2(\R^3\times\R^3)}\lesssim \|g_n\|_{L^2(\Omega\times\R^3_v)}+ \|\Lambda g_n\|_{L^2(\Omega\times\R^3)}. 
\end{align}
This gives the regularizing effect of $\Lambda$ with respect to spatial variable. 
In view of \eqref{222b} and \eqref{222a}, noticing that $K$ is an operator acted on velocity only and supported on $|v|\le 2R$ while $E$ is the extension operator on $g_n$, we have 
\begin{align*}
	\|Kg_n\|_{H^{\frac{s_1}{2}}(\Omega)H^{\frac{s}{2}}(|v|\le 2R)}&\le \|\<D_x\>^{\frac{s_1}{2}}\<D_v\>^{\frac{s}{2}}KEg_n\|_{L^2(\R^3\times\R^3)}\\
	&\le \|\<D_x\>^{s_1}KEg_n\|^{1/2}_{L^2(\R^3\times\R^3)}\|\<D_v\>^{s}KEg_n\|^{1/2}_{L^2(\R^3\times\R^3)}\\
	&\lesssim \|g_n\|_{L^2(\Omega\times\R^3_v)}+ \|\Lambda g_n\|_{L^2(\Omega\times\R^3)}. 
\end{align*}
By this estimate, since $\{g_n\}$ and $\{\Lambda g_n\}$ are bounded in $L^2(\Omega\times\R^3)$, it follows by Sobolev embedding that $\{Kg_n\}$ contains a convergent subsequence. This implies that $K$ is $\Lambda$-compact.

\smallskip
	Using \cite[Corollary XVII.4.4]{Gohberg1990} for stability of essential spectrum under relatively compact perturbations, we proved that $\sigma_{ess}(\mathcal{L})\subset \sigma_{ess}(\Lambda)\subset \{z\in\C:\Re z\le -c_0\}$, where the later inclusion follows from Lemma \ref{thm.spex}. Then by \cite[Theorem XVII.2.1]{Gohberg1990}, we know that $\sigma(\mathcal{L})\cap \{\Re z>-c_0\}$ contains only discrete eigenvalues of finite type.
	On the other hand, similar to \eqref{eq1}, $\mathcal{L}$ is non-positive, i.e. $(\mathcal{L}f,f)_{L^2_{x,v}}\le 0$. Then the discrete spectrum of $\mathcal{L}$ lies in $(-\infty,0]$. This proves that $\sigma(\mathcal{L})\cap \{z\in \C:\Re z>-c_0\}\subset \{z\in\R:-c_0< z\le 0\}$.
	
	\medskip 
	\noindent{\bf Step 2.} To prove that the kernel of $\L$ in $X$ is $\{0\}$,
	 we let $f\in \ker \L$. Then 
	\begin{align*}
		(v\cdot\na_xf,f)_{L^2_{x,v}}+(-Lf,f)_{L^2_{x,v}}=0. 
	\end{align*}
Using boundary condition in \eqref{X1} and non-negativity of $-L$, we have 
\begin{align*}
	\int_{\pa\Omega}\int_{v\cdot n>0}v\cdot n|f|^2\,dvdS(x)=(-Lf,f)_{L^2_{x,v}}=0. 
\end{align*}
This implies that $f|_{\gamma_-}=0$ and $f=\P f$. Hence, 
	\begin{align*}
		f(x,v) = \big(a(x)+b(x)\cdot v+c(x)|v|^2\big)\mu^{1/2}(v),
	\end{align*}for some functions $a(x)$, $b(x)$ and $c(x)$ depending possibly on $x$. 
	Since $\L f=0$ and $f=\P f$, we have 
	\begin{align*}
		v\cdot\nabla_x \big(a(x)+b(x)\cdot v+c(x)|v|^2\big)\mu^{1/2}(v) =0,
	\end{align*}
and thus, $\nabla_x(a,b,c)(x) =0$, in the sense of distributions. It follows that $a=b=c=0$ since $f=0$ on the boundary $\gamma_-$.  This completes the proof of Theorem \ref{Main1}.
\end{proof}

\section{Exponential decay for linear equation}\label{sec4}

In this section, we will derive the exponential decay for linearized equation corresponding to \eqref{1}. We consider the linear equation 
\begin{align}\label{11a}
	\partial_tf  + v\cdot\nabla_xf  =  Lf,\quad f(0,x,v)=f_0(x,v), 
\end{align}
with the inflow boundary condition 
\begin{align}\label{11in}
f(t,x,v) = g(t,x,v),\quad \text{ on } \gamma_-.
\end{align}

To find the macroscopic dissipation, we take the following velocity moments
\begin{equation*}
	\mu^{\frac{1}{2}}, v_j\mu^{\frac{1}{2}}, \frac{1}{6}(|v|^2-3)\mu^{\frac{1}{2}},
	(v_j{v_m}-1)\mu^{\frac{1}{2}}, \frac{1}{10}(|v|^2-5)v_j \mu^{\frac{1}{2}}
\end{equation*}
with {$1\leq j,m\leq 3$} for the equation \eqref{11a}. One sees that  
the coefficient functions $[a,b,c]=[a,b,c](t,x)$ in \eqref{Pf} satisfy the fluid-type system 
\begin{equation}\label{11}
	\left\{\begin{array}{l}
		\dis \pa_t a +\nabla_x \cdot b=0,\\
		\dis \pa_t b +\na_x (a+2c)+\na_x\cdot \Theta (\{\I-\P\} f)=0,\\[1mm]
		\dis \pa_t c +\frac{1}{3}\na_x\cdot b +\frac{1}{6}\na_x\cdot
		\Lambda (\{\I-\P\} f)=0,\\[2mm]
		\dis \pa_t[\Theta_{{ jm}}(\{\I-\P\} f)+2c\de_{{ jm}}]+\pa_jb_m+\pa_m
		b_j=\Theta_{jm}({r}+{h}),\\[2mm]
		\dis \pa_t \Lambda_j(\{\I-\P\} f)+\pa_j c = \Lambda_j({r}+{h}),
	\end{array}\right.
\end{equation}
where the
high-order moment functions $\Theta=(\Theta_{jm})_{3\times 3}$ and
$\Lambda=(\Lambda_j)_{1\leq j\leq 3}$ are respectively defined by
\begin{equation}
	\Theta_{jm}(f) = \left ((v_jv_m-1)\mu^{\frac{1}{2}}, f\right)_{L^2_v},\ \ \
	\Lambda_j(f)=\frac{1}{10}\left ((|v|^2-5)v_j\mu^{\frac{1}{2}},
	f\right)_{L^2_v},\notag
\end{equation}
with the inner product taken with respect to velocity variable $v$ only, and the terms ${r}$ and ${h}$ on the right are given by
\begin{equation*}
	{r}= -{v}\cdot \na_{{x}} \{\I-\P\}f,\ \ {h}=L \{\I-\P\}f+g.
\end{equation*}
Notice that system \eqref{11} is understood in the sense of distributions. 

\begin{Lem}\label{LemMarco}
	Let $f$ be a solution to \eqref{11a} with inflow boundary condition \eqref{11in}. Then there exists an instant energy function $\E_{int}(t)$ satisfying 
	\begin{align}\label{Eint}
		|\E_{int}(t)|\lesssim \|f\|_{L^2_xL^2_v}
	\end{align}
such that 
	\begin{multline}\label{marco}
		\partial_t\E_{int}(t) + \lambda\|[{a},{b},{c}]\|^2_{L^2_{x}}
		\lesssim \|\{\I-\P\}{ f}\|_{L^2_{x}L^2_D}\\+\int_{\partial\Omega}\int_{v\cdot n>0}|v\cdot n||f|^2\,dvdS(x) + \int_{\partial\Omega}\int_{v\cdot n<0}|v\cdot n||g|^2\,dvdS(x).
	\end{multline}
\end{Lem}
\begin{proof}

	Let ${\Phi}(t,x,v)\in C^1((0,+\infty)\times\Omega\times\R^3)$ be a test function. Taking the inner product of ${\Phi}(t,x,v)$ and \eqref{11a} with respect to $(x,v)$, we obtain 
	\begin{align*}
		\partial_t(f,{\Phi})_{L^2_{x,v}}(t)&- (f,\partial_t{\Phi})_{L^2_{x,v}}-(f,v\cdot{\nabla_{x}\Phi})_{L^2_{x,v}} 
		\\
		&+\int_{\partial\Omega}(v\cdot n(x)f(x),{\Phi}(x))_{L^2_v}\,dS(x) = (L f,{\Phi})_{L^2_{x,v}},
	\end{align*}
	where $dS(x)$ is the spherical measure. 
	Using the decomposition ${f}=\P{f}+\{\I-\P\}{f}$, we have 
	\begin{align}\label{100}
		\partial_t(f,{\Phi})_{L^2_{x,v}}(t)-({\P f},v\cdot{\nabla_{x}\Phi})_{L^2_{x,v}}  = \sum_{j=1}^4S_j,
	\end{align}
	where $S_j$ are defined by 
	\begin{align*}
		S_1 &= (f,\partial_t{\Phi})_{L^2_{x,v}},\\
		S_2 &= ({\{\I-\P\}f},v\cdot{\nabla_{x}\Phi})_{L^2_{x,v}} ,\\
		S_3&= (L f,{\Phi})_{L^2_{x,v}},\\
		S_4 &= -\int_{\partial\Omega}(v\cdot n(x)f(x),{\Phi}(x))_{L^2_v}\,dS(x).
	\end{align*}

\medskip \noindent{\bf Estimate on ${c}(t,x)$:} We choose the test function 
$$
\Phi= \Phi_c= (|v|^2-5)v\cdot\nabla_{x}\phi_c(t,x)\mu^{1/2},
$$
where 				
\begin{equation*}\left\{\begin{aligned}
			&-\Delta_x \phi_c = {c}\ \text{ in } \Omega,\\
			&{\phi_c}(x)= 0 \ \text{ on }\ \pa\Omega.
		\end{aligned}\right.
	\end{equation*}
Then by standard elliptic estimates, one has 
\begin{align}\label{35}
	\|\phi_c\|_{H^2_x(\Omega)}\le \|c\|_{L^2_x}, \quad \|\pa_t\phi_c\|_{H^1_x(\Omega)}\le \|\pa_tc\|_{H^{-1}_x}. 
\end{align}
Using the third equation of \eqref{11}, we have 
\begin{align}\label{36}
	\|\pa_t\phi_c\|_{H^1_x}\le \|\pa_tc\|_{H^{-1}_x}\lesssim \|b\|_{L^2_x}+\|(\I-\P)f\|_{L^2_xL^2_D}.  
\end{align}
For the second term on the left-hand side of \eqref{100}, we have 
	\begin{align*}
		&\quad\,-({\P f},v\cdot{\nabla_{x}\Phi_{c}})_{L^2_{x,v}} \\
		&= -\sum_{j,m=1}^3({a}+{b}\cdot v+\frac{1}{2}(|v|^2-3){c} ,v_jv_m(|v|^2-5)\mu{\partial_j\partial_m\phi_{c}})_{L^2_{x,v}} \\
		&= 5\sum_{j=1}^3({c} ,{-\partial^2_j\phi_{c}})_{L^2_{x,v}}  = 5\|{c}\|^2_{L^2_{x,v}} .
	\end{align*}
	Note that $\int_{\R^3}|v|^4v^2_j\mu\,dv=35$, $\int_{\R^3}|v|^2v^2_j\mu\,dv=5$ and  $\int_{\R^3}v^2_j\mu\,dv=1$.  
	For $S_1$, we deduce from \eqref{36} that 
	\begin{align*}
		|S_1|&\le|(f,\partial_t{\Phi_c})_{L^2_{x,v}}| = |(\{\I-\P\}f,\partial_t{\Phi_c})_{L^2_{x,v}}|\\
		&\lesssim \eta\|\partial_t{\nabla_x\phi_c}\|^2_{L^2_{x}}+C_\eta\|\{\I-\P\}f\|^2_{L^2_{x}L^2_D}\\
		&\lesssim \eta\|b\|^2_{L^2_{x}}+C_\eta \|\{\I-\P\}f\|^2_{L^2_{x}L^2_D},
	\end{align*}
for any $\eta>0$. 
	For $S_2$, we have from \eqref{35} that  
	\begin{align*}
		|S_2|&\lesssim \eta \|\phi_c\|_{H^2_x}^2+C_\eta\|\{\I-\P\}f\|^2_{L^2_xL^2_D}\\&\lesssim \eta\|{c}\|^2_{L^2_{x}}+C_\eta\|\{\I-\P\}f\|^2_{L^2_{x}L^2_D}.
	\end{align*}
	For $S_3$, applying \eqref{e22} and \eqref{35}, we have 
	\begin{align*}
		|S_3|\le \eta \|{c}\|^2_{L^2_{x}}+C_\eta \|\{\I-\P\}f\|^2_{L^2_{x}L^2_D}.
	\end{align*}
	For $S_4$, by trace theorem, we have 
	\begin{align*}
		\int_{\pa\Omega}|\na_x\phi_c|^2\,dS(x) \lesssim \|\phi_c\|_{H^2_x(\Omega)}^2 \lesssim \|c\|_{L^2_x}^2,
	\end{align*}
and hence 
	\begin{align}\label{s4}\notag
		|S_4|&\le C_\eta\int_{\partial\Omega}\int_{\R^3}|v\cdot n||f|^2\,dvdS(x) + \eta\int_{\pa\Omega}|\na_x\phi_c|^2\,dS(x)\\
		&\lesssim  C_\eta\int_{\partial\Omega}\int_{v\cdot n>0}|v\cdot n||f|^2\,dvdS(x) + C_\eta\int_{\partial\Omega}\int_{v\cdot n<0}|v\cdot n||g|^2\,dvdS(x) + \eta\|c\|_{L^2_x}^2.
	\end{align}
	Collecting the above estimates for $S_j$ $(1\le j\le 4)$ and letting $\eta>0$ be suitably small, we obtain
	\begin{multline}\label{122a}
		\partial_t(f,\Phi_c)_{L^2_{x,v}} + \lambda\|{c}\|^2_{L^2_{x}} 
		\lesssim \eta \|{b}\|^2_{L^2_{x}}
		+C_\eta\|\{\I-\P\}f\|^2_{L^2_{x}L^2_D} \\
		+C_\eta\int_{\partial\Omega}\int_{v\cdot n>0}|v\cdot n||f|^2\,dvdS(x) + C_\eta\int_{\partial\Omega}\int_{v\cdot n<0}|v\cdot n||g|^2\,dvdS(x),
	\end{multline}for some $\lambda>0$.

	\medskip \noindent{\bf Estimate of ${b}(t,x)$:}
	For the estimates of $b$, we choose 
	\begin{align*}
		{\Phi}={\Phi_b}=\sum^3_{m=1}{\Phi^{j,m}_b},\ j=1,2,3,
	\end{align*}
	where 
	\begin{equation*}
		{\Phi^{j,m}_b}=\left\{\begin{aligned}
			\big(|v|^2v_mv_j{\partial_{x_m}\phi_j}-\frac{7}{2}(v_m^2-1){\partial_{x_j}\phi_j}\big)\mu^{1/2},\ m\neq j,\\
			\frac{7}{2}(v_j^2-1){\partial_{x_j}\phi_j}\mu^{1/2},\qquad\qquad\qquad m=j,
		\end{aligned}\right.
	\end{equation*}
	and $\phi_j$($1\le j\le 3$) solves  
	\begin{equation*}\left\{\begin{aligned}
			&-\Delta_x \phi_j = {b_j}\ \text{ in }\Omega,\\
			&{\phi_j}(x) = 0 \ \text{ on }\pa\Omega.
		\end{aligned}\right.
	\end{equation*}
Similar to \eqref{35} and \eqref{36}, we deduce from the second equation of \eqref{11} that 
\begin{align}\label{35a}
	\|\phi_j\|_{H^2_x(\Omega)}\le \|b\|_{L^2_x}, \quad 
	\|\pa_t\phi_j\|_{H^1_x}\le \|\pa_tb\|_{H^{-1}_x}\lesssim \|[a,c]\|_{L^2_x}+\|(\I-\P)f\|_{L^2_xL^2_D}.  
\end{align}
In view of this, we have the following estimates. 
	For $S_1$, one has 
	\begin{align*}\notag
		\dis|S_1| &\le  \big(\P f,\partial_t\Phi_b\big)_{L^2_{x,v}}+\big(\{\I-\P\} f,\partial_t\Phi_b\big)_{L^2_{x,v}}\\[1mm]
		&\notag\lesssim C_\eta\|c\|_{L^2_x}^2 + C_\eta\|\{\I-\P\} f\|_{L^2_x}^2 + \eta\|\partial_t\na_x b_j\|_{L^2_x}^2\\[1mm]
		&\lesssim  C_\eta\|c\|_{L^2_x} +C_\eta\|\{\I-\P\}f\|_{L^2_x} +\eta\|a\|_{L^2_x}.
	\end{align*}	
The term $S_2$ can be estimated as 
\begin{align*}
|S_2|\lesssim \|\{\I-\P\}f\|_{L^2_xL^2_D}\sum_{i,j,m=1}^3\|\partial_{x_ix_m}\phi_j\|_{L^2_x}
\lesssim C_\eta\|\{\I-\P\}f\|^2_{L^2_xL^2_D}+\eta \|b\|^2_{L^2_x}.
\end{align*}
	For $S_3$, we have from \eqref{e22} that 
	\begin{align*}\notag
		|S_3|
		&\lesssim\notag C_\eta\|\{\I-\P\}f\|^2_{L^2_xL^2_D}+ \eta\|\nabla_x\phi_j\|_{L^2_x}\\
		&\lesssim C_\eta\|\{\I-\P\}f\|_{L^2_xL^2_D}^2+ \eta\|b\|_{L^2_x}.
	\end{align*}
	For $S_4$, similar to \eqref{s4}, by trace theorem, we have 
\begin{align*}
	\int_{\pa\Omega}|\na_x\phi_j|^2\,dS(x) \lesssim \|\phi_j\|_{H^2_x(\Omega)}^2 \lesssim \|b\|_{L^2_x}^2,
\end{align*}
and hence 
\begin{align*}
	|S_4|\lesssim  C_\eta\int_{\partial\Omega}\int_{v\cdot n>0}|v\cdot n||f|^2\,dvdS(x) + C_\eta\int_{\partial\Omega}\int_{v\cdot n<0}|v\cdot n||g|^2\,dvdS(x) + \eta\|b\|_{L^2_x}^2.
\end{align*}
	For the second term on the left-hand side of \eqref{100}, we have  
	\begin{align*}\notag
		&\quad\,-\sum^3_{m=1}(\P f,v\cdot \na_x{\Phi^{j,m}_b})_{L^2_{x,v}}\\
		&\notag=-\sum^3_{m=1}\big(({a}+{b}\cdot v+\frac{1}{2}(|v|^2-3){c})\mu^{1/2},v\cdot{\nabla_x\Phi^{j,m}_b}\big)_{L^2_{x,v}}\\
		\notag&=-\sum^{3}_{m=1,m\neq j}(v_mv_j\mu^{1/2}{b_j},|v|^2v_mv_j\mu^{1/2}{\partial_{x_m}^2\phi_j})_{L^2_{x,v}}\\
		&\notag\qquad-\sum^{3}_{m=1,m\neq j}(v_mv_j\mu^{1/2}{b_m},|v|^2v_mv_j\mu^{1/2}{\partial_{x_m}\partial_{x_j}\phi_j})_{L^2_{x,v}}\\
		&\notag\qquad+7\sum^{3}_{m=1,m\neq j}({b_m},{\partial_{x_m}\partial_{x_j}\phi_j})_{L^2_{x}}-7({b_m},{\partial^2_{x_j}\phi_j})_{L^2_{x}}\\
		&= -7 \sum^3_{m=1}({b_j},{\partial_{x_m}^2\phi_j})_{L^2_{x}}=7\|{b_j}\|^2_{L^2_{x}}.
	\end{align*}
Combining the above estimates and letting $\eta>0$ be  sufficiently small, we have 
	\begin{multline}\label{122b}
		\partial_t(f,\Phi_b)_{L^2_{x,v}} + \lambda\|{b}\|^2_{L^2_{x}}\lesssim \eta\|a\|^2_{L^2_{x}}+ C_\eta\|{c}\|^2_{L^2_{x}}
		+C_\eta\|\{\I-\P\}{f}\|^2_{L^2_{x}L^2_D}\\  +C_\eta\int_{\partial\Omega}\int_{v\cdot n>0}|v\cdot n||f|^2\,dvdS(x) + C_\eta\int_{\partial\Omega}\int_{v\cdot n<0}|v\cdot n||g|^2\,dvdS(x),
	\end{multline}for some $\lambda>0$.

	\medskip\noindent{\bf Estimate on ${a}(t,x)$:} We choose the following test function
	\begin{align*}
		{\Phi} = {\Phi_{a}} = (|v|^2-10)v\cdot{\nabla_{x}\phi_{a}}(t,x)\mu^{1/2},
	\end{align*}
	where $\phi_a$ solves 
	\begin{equation*}\left\{\begin{aligned}
			&-\Delta_x \phi_a = {a}\ \text{ in }\Omega,\\
			&{\phi_a}(x)= 0 \ \text{ on }\ \pa\Omega.
		\end{aligned}\right.
	\end{equation*}
Similar to \eqref{35} and \eqref{36}, we deduce from the first equation of \eqref{11} and standard elliptic estimates that 
\begin{align}\label{35b}
	\|\phi_a\|_{H^2_x(\Omega)}\le \|a\|_{L^2_x}, \quad 
	\|\pa_t\phi_a\|_{H^1_x}\le \|\pa_ta\|_{H^{-1}_x}\lesssim \|b\|_{L^2_x}.  
\end{align}
For the second term on the left-hand side of \eqref{100}, we have 
	\begin{align*}
		&\quad\,-({\P f},v\cdot{\nabla_{x}\Phi_{a}})_{L^2_{x,v}} \\
		&= -\sum_{j,m=1}^3({a}+{b}\cdot v+\frac{1}{2}(|v|^2-3){c} ,v_jv_m(|v|^2-10)\mu{\partial_j\partial_m\phi_{a}})_{L^2_{x,v}} \\
		&= \sum_{j=1}^3({a} ,{-\partial^2_j\phi_{a}})_{L^2_{x}}  = \|{a}\|^2_{L^2_{x}} .
	\end{align*}
For $S_1$, we apply \eqref{35b} to obtain 
	\begin{align*}
		|S_1|&\le \big|\big(\{\I-\P\}f, \pa_t\Phi_a\big)_{L^2_{x,v}}\big|
		+\big|\big(\P f, \pa_t\Phi_a\big)_{L^2_{x,v}}\big|\\
		&\lesssim \|\{\I-\P\}f\|^2_{L^2_xL^2_D} + \|b\|_{L^2_x}^2 + \|\pa_t\na_x\phi_a\|^2_{L^2_x}\\
		&\lesssim \|\{\I-\P\}f\|^2_{L^2_xL^2_D} + \|b\|_{L^2_x}^2. 
	\end{align*}	
	For $S_2$ and $S_3$, by \eqref{35b}, we have 
	\begin{align*}
		|S_2|+|S_3|\lesssim C_\eta\|\{\I-\P\}f\|^2_{L^2_xL^2_D} + \eta\|a\|^2_{L^2_x}. 
	\end{align*}
For $S_4$, similar to \eqref{s4}, by trace theorem and Cauchy-Schwarz inequality, we have 
\begin{align*}
	|S_4|\lesssim  C_\eta\int_{\partial\Omega}\int_{v\cdot n>0}|v\cdot n||f|^2\,dvdS(x) + C_\eta\int_{\partial\Omega}\int_{v\cdot n<0}|v\cdot n||g|^2\,dvdS(x) + \eta\|a\|_{L^2_x}^2.
\end{align*}
	Combining the above estimates and letting $\eta>0$ be small enough, we have 
	\begin{multline}\label{122d}
		\partial_t(f,\Phi_a)_{L^2_{x,v}} + \lambda\|{a}\|^2_{L^2_{x}}
		\lesssim \|\partial^\alpha\{\I-\P\}{f}\|^2_{L^2_{x}L^2_D}+\|{b}\|^2_{L^2_{x}}\\
		 +C_\eta\int_{\partial\Omega}\int_{v\cdot n>0}|v\cdot n||f|^2\,dvdS(x) + C_\eta\int_{\partial\Omega}\int_{v\cdot n<0}|v\cdot n||g|^2\,dvdS(x).
	\end{multline}

	\medskip
	
	Now we take the linear combination $\eqref{122a}+\kappa\times\eqref{122b}+\kappa^2\times\eqref{122d}$,  and let $\kappa,\eta>0$ be sufficiently small to obtain  
	\begin{multline*}
		\partial_t\E_{int}(t) + \lambda\|[{a},{b},{c}]\|^2_{L^2_{x}}
		\lesssim \|\{\I-\P\}{ f}\|_{L^2_{x}L^2_D}\\+\int_{\partial\Omega}\int_{v\cdot n>0}|v\cdot n||f|^2\,dvdS(x) + \int_{\partial\Omega}\int_{v\cdot n<0}|v\cdot n||g|^2\,dvdS(x),
	\end{multline*}
	where $\E_{int}(t)$ is given by 
	\begin{align*}
		\E_{int}(t) = 
		(f,\Phi_c)_{L^2_{x,v}} + \kappa(f,\Phi_b)_{L^2_{x,v}}+\kappa^2(f,\Phi_a)_{L^2_{x,v}}.
	\end{align*}
	Using \eqref{35}, \eqref{35a} and \eqref{35b}, we know that 
	\begin{align*}
		|\E_{int}(t)|\lesssim \|f\|_{L^2_xL^2_v}. 
	\end{align*}
This completes the proof of Lemma \ref{LemMarco}.	
\end{proof}

Applying the above macroscopic estimates on $[a,b,c]$, we can prove Theorem \ref{MaindecayLandau}. 
\begin{proof}
	[Proof of Theorem \ref{MaindecayLandau}]
	We are going to prove the exponential decay of solution to linearized Landau and Boltzmann equations. 
	Let $W$ be given by \eqref{W}. 
	Multiplying \eqref{11a} by $W$ and noticing \eqref{W1}, we have 
	\begin{align}\label{37}
		\pa_t(Wf) + v\cdot \na_x (Wf) + q|v|^2\<v\>^{-1}Wf - WL(f)  = 0.
	\end{align}
Letting $q=0$ in \eqref{W} and taking the inner product of \eqref{37} with $f$ over $\Omega\times\R^3_v$ to obtain
\begin{multline}\label{38a}
	\frac{1}{2}\pa_t\|f\|^2_{L^2_{x,v}} + \frac{1}{2}\int_{\pa\Omega}\int_{v\cdot n>0}v\cdot n|f|^2\,dvdS(x) + c_1\|(\I-\P)f\|^2_{L^2_xL^2_D}\\ \le  \frac{1}{2}\int_{\pa\Omega}\int_{v\cdot n<0}|v\cdot n||g|^2\,dvdS(x). 
\end{multline}
Here we used \eqref{e21}. Letting $q>0$ in \eqref{W} and taking the inner product of \eqref{37} with $Wf$ over $\Omega\times\R^3_v$ to obtain
\begin{multline}\label{38b}
	\frac{1}{2}\pa_t\|Wf\|^2_{L^2_{x,v}} + \frac{1}{2}\int_{\pa\Omega}\int_{v\cdot n>0}v\cdot n|Wf|^2\,dvdS(x) + q\||v|\<v\>^{-\frac{1}{2}}Wf\|_{L^2_{x,v}}
	\\
	\le C_1\|f\|^2_{L^2_xL^2_D} + \frac{1}{2}\int_{\pa\Omega}\int_{v\cdot n<0}|v\cdot n||Wg|^2\,dvdS(x). 
\end{multline}
Taking the linear combination $\eqref{38a}+\kappa\times\eqref{marco}$ with sufficiently small $\kappa>0$, we have 
\begin{multline}\label{38c}
	\pa_t\Big(\frac{1}{2}\|f\|^2_{L^2_{x,v}}+ \kappa\E_{int}(t)\Big) + \lambda\|f\|^2_{L^2_xL^2_D}
	+\frac{1}{4}\int_{\partial\Omega}\int_{v\cdot n>0}v\cdot n|f|^2\,dvdS(x) \\
	\lesssim  \int_{\pa\Omega}\int_{v\cdot n<0}|v\cdot n||g|^2\,dvdS(x),
\end{multline}
for some constant $\lam>0$. 
To find the large velocity decay and eliminate the term $\|f\|_{L^2_xL^2_D}$ in \eqref{38b}, we take the combination $\eqref{38c}+\kappa\times\eqref{38b}$ with sufficiently small $\kappa>0$ to deduce that 
\begin{multline}\label{38d}
	\pa_t\E(t)
 + \lambda\|f\|^2_{L^2_xL^2_D}
	+\frac{1}{4}\int_{\partial\Omega}\int_{v\cdot n>0}v\cdot n|f|^2\,dvdS(x) \\
	 + \frac{\kappa}{2}\int_{\pa\Omega}\int_{v\cdot n>0}v\cdot n|Wf|^2\,dvdS(x) + q\kappa\||v|\<v\>^{-\frac{1}{2}}Wf\|_{L^2_{x,v}}
	\\
	\lesssim  \int_{\pa\Omega}\int_{v\cdot n<0}|v\cdot n||g|^2\,dvdS(x),
\end{multline}
with some small constant $\lam>0$. Here $\E(t)$ is defined by 
\begin{align*}
	\E(t) = \frac{1}{2}\|f\|^2_{L^2_{x,v}}+ \kappa\E_{int}(t)+\frac{\kappa}{2}\|Wf\|^2_{L^2_{x,v}}.
\end{align*}
Choosing $\kappa>0$ small enough, and using \eqref{Eint} and noticing $W\approx 1$, we know that 
\begin{align*}
	\E(t) \approx \|f\|_{L^2_{x,v}}^2. 
\end{align*}
Noticing that 
\begin{align*}
	\|f\|_{L^2_xL^2_v}\lesssim \|f\|^2_{L^2_xL^2_D} + \||v|\<v\>^{-\frac{1}{2}}Wf\|_{L^2_{x,v}},
\end{align*}
we deduce from \eqref{38d} that 
\begin{align}\label{38e}
	\pa_t\E(t) + \lam \E(t) \lesssim \int_{\pa\Omega}\int_{v\cdot n<0}|v\cdot n||g|^2\,dvdS(x).
\end{align}
This is the main {\em a priori} estimate of linear equation \eqref{11a}. Solving ODE \eqref{38e}, we obtain that 
\begin{align*}
	\sup_{t\ge 0}\{e^{\lam t}\E(t)\} \lesssim \|f_0\|^2_{L^2_{x,v}} + \sup_{t\geq 0}\int_{\pa\Omega}\int_{v\cdot n<0}|v\cdot n|e^{\lam t}|g|^2\,dvdS(x).
\end{align*}
Choosing $0<\lam<\delta_0$ with $\delta_0$ given in \eqref{g1}, we conclude Theorem \ref{MaindecayLandau}. 
\end{proof}

\section{Exponential decay for cutoff Boltzmann equation}\label{sec5}

In this section, we will derive the exponential decay for cutoff Boltzmann equation with soft potential. 
We begin with denoting the following boundary integral.
For fixed $x\in\pa\Omega$, denote the boundary inner product over $\gamma_\pm$ in $v$ as 
\begin{align*}
	(g_1,g_2)_{\gamma_\pm}(t,x) = \int_{\pm v\cdot n(x)>0}g_1(t,x,v)g_2(t,x,v)|v\cdot n(x)|\,dv. 
\end{align*}
We define $\|g\|_{\gamma}=\|g\|_{\gamma_+}+\|g\|_{\gamma_-}$ to be the $L^2(\gamma)$ with respect to the measure $|v\cdot n(x)|dS(x)dv$. Here $\|g\|_{\gamma_\pm}^2=(g,g)_{\gamma_\pm}$. 

\begin{Lem}\label{Lem42}
	There exists $M>0$ such that for any solution $f(t,x,v)$ to the linearized Boltzmann/Landau equation \eqref{1}, 
	\begin{align}\label{44}
		\int^1_0\|\P f(s)\|_{L^2_xL^2_D}^2\,ds\le M\Big\{c_1\int^1_0\|(\I-\P) f(s)\|_{L^2_xL^2_D}^2\,ds+ \int^1_0\int_{\pa\Omega}\int_{\R^3}|v\cdot n||f(s)|^2dvdS(x)ds\Big\}, 
	\end{align}
where $c_1>0$ is a constant given in \eqref{e21}. 
\end{Lem}
	The proof of Lemma \ref{Lem42} is based on a contradiction argument.
	 If Lemma \ref{Lem42} was false, then no $M$ exists as in Lemma \ref{Lem42} for any solution $f$. Hence, for any $k\ge 1$, there exists a sequence of non-zero solutions $f_k(t,x,v)$ to the linearized Boltzmann/Landau equation \eqref{1} such that 
	\begin{align*}
		\int^1_0\|\P f_k(s)\|_{L^2_xL^2_D}^2\,ds\ge k\Big\{c_1\int^1_0\|(\I-\P) f_k(s)\|_{L^2_xL^2_D}^2\,ds+ \int^1_0\int_{\pa\Omega}\int_{\R^3}|v\cdot n||f_k(s)|^2dvdS(x)ds\Big\}.
	\end{align*}
	Equivalently, in terms of normalization $Z_k(t,x,v)\equiv \frac{f_k(t,x,v)}{\sqrt{\int^1_0\|\P f_k(s)\|_{L^2_xL^2_D}^2\,ds}}$, we have 
	\begin{align}\label{48}
		\int^1_0\|\P Z_k(s)\|_{L^2_xL^2_D}^2\,ds = 1,
	\end{align}
and 
\begin{align}\label{49}
	\Big\{c_1\int^1_0\|(\I-\P) Z_k(s)\|_{L^2_xL^2_D}^2\,ds+ \int^1_0\int_{\pa\Omega}\int_{\R^3}|v\cdot n||Z_k(s)|^2dvdS(x)ds\Big\} \le \frac{1}{k}.
\end{align}
We also have from $\{\pa_t+v\cdot\na_x+L\}f_k=0$ that 
\begin{align}\label{49a}
	\{\pa_t+v\cdot\na_x+L\}Z_k=0.
\end{align}
Since $\sup_{k\ge 1}\int^1_0\|Z_k(s)\|_{L^2_xL^2_D}^2\,ds<\infty$, up to a subsequence, there exists a function $Z(t,x,v)$ satisfying $\int^1_0\|Z\|^2_{L^2_xL^2_D}ds<\infty$ such that 
\begin{align*}
	Z_k\rightharpoonup Z\ \text{ weakly in }\int^1_0\|\cdot\|_{L^2_xL^2_D}^2\,ds,
\end{align*}
and from \eqref{49}, 
\begin{align}\label{49b}
	\int^1_0\|(\I-\P) Z_k(s)\|_{L^2_xL^2_D}^2\,ds\to 0, \ \text{ as }k\to \infty. 
\end{align}
Then it is straightforward to verify  that 
\begin{align*}
	\P Z_k\rightharpoonup \P Z\ \text{ and }\ (\I-\P)Z_k\rightharpoonup(\I-\P) Z, \ \text{ weakly in }\int^1_0\|\cdot\|_{L^2_xL^2_D}^2\,ds, 
\end{align*}
and $(\I-\P)Z=0$ from \eqref{49b}. Thus, 
\begin{align*}
	Z(t,x,v) = \{a(t,x)+v\cdot b(t,x) +|v|^2c(t,x)\}\sqrt{\mu}. 
\end{align*}
Also, we deduce from \eqref{49a} and \eqref{49b} that 
\begin{align}\label{49c}
	\pa_t Z + v\cdot \na_x Z = 0. 
\end{align}
Now the main strategy is to show that, on one hand, $Z=0$ from \eqref{49c} and the inherited boundary conditions \eqref{49}. On the other hand, $Z_k$ will be shown to converge strongly to $Z$ in $\int^1_0\|\cdot\|_{L^2_xL^2_D}^2\,ds$, and $\int^1_0\|Z\|_{L^2_xL^2_D}^2\,ds\neq 0$. This leads to a contradiction.

	\begin{Lem}\label{Lem43}
		There exist constants $a_0,c_0,c_1,c_2$, and constant vectors $b_0,b_1$ and $\varpi$ such that $Z(t,x,v)$ takes the form:
		\begin{align}\label{499}
			\bigg(\Big\{\frac{c_0}{2}|x|^2-b_0\cdot x+a_0\Big\}+\big\{-c_0tx-c_1x+\varphi\times x+b_0t+b_1\big\}\times v+\Big\{\frac{c_0t^2}{2}+c_1t+c_2\Big\}|v|^2\bigg)\sqrt\mu. 
		\end{align}
	Moreover, these constants are finite:
	\begin{align*}
		|a_0|+|c_0|+|c_1|+|c_2|+|b_0|+|b_1|+|\varpi|<\infty.
		\end{align*}
	\end{Lem}	
	\begin{proof}
		See Lemma 6 in \cite[pp.  736]{Guo2009}. 
	\end{proof}	
		
	\begin{Lem}\label{Lem44}
		For any smooth function $\chi_0(t,x)$ such that $\text{supp} \chi_0\subset\subset \Omega$, one has, up to a subsequence,
		 \begin{align}\label{412b}
			\lim_{k\to \infty}\|\chi_0\{Z_k-Z\}\|_{L^2_xL^2_D}^2 =0.
		\end{align}
	\end{Lem}
\begin{proof}
	Choose any $\eta_0>0$ and a smooth cutoff function $\chi_1(t,v)$ in $(0,1)\times\R^3_v$ such that $\chi_1(t,v)=1$ in $[\eta_0,1-\eta_0]\times\{|v|\le \frac{1}{\eta_0}\}$. Multiplying $\chi_0\chi_1$ with \eqref{49a}, we obtain 
	\begin{align*}
		[\pa_t+v\cdot\na_x]\{\chi_0\chi_1Z_k\} = \{[\pa_t+v\cdot \na_x]\chi_0\chi_1\}Z_k - \chi_0\chi_1LZ_k.
	\end{align*}
Since $\int^1_0\|Z_k\|^2_{L^2_xL^2_D}\,ds<\infty$, one has $\chi_0\chi_1Z_k\in L^2([0,1]\times\Omega\times\R^3))$ and $\{[\pa_t+v\cdot \na_x]\chi_0\chi_1\}Z_k - \chi_0\chi_1LZ_k\in L^2([0,1]\times\Omega\times\R^3))$. Then we deduce from the averaging lemma \cite{Diperna1989,Diperna1989a} that $$\int_{\R^3}\chi_0\chi_1Z_k(v)\phi(v)\,dv$$ are compact in $L^2([0,1]\times\Omega)$ for any smooth cutoff function $\phi(v)$. It then follows that 
\begin{align}\label{410}
	\int_{\R}\chi_0\chi_1Z_k(v)[1,v,|v|^2]\sqrt\mu\,dv
\end{align}
are compact in $L^2([0,1]\times\Omega)$. 
	On the other hand, using a similar argument to that for obtaining Lemma 8 in \cite[pp. 340]{Guo2003}, one has from \eqref{49a} that 
	\begin{align*}
		\sup_{0\le t\le 1}\|Z_k(t)\|_{L^2_xL^2_D}^2 &\lesssim \|Z_k(0)\|^2_{L^2_xL^2_D},\\
		\int^1_0\|Z_k(s)\|^2_{L^2_xL^2_D}\,ds&\ge C\|Z_k(0)\|_{L^2_xL^2_D}^2. 
		\end{align*}
	In view of this, \eqref{48} and \eqref{49}, for any exponentially decay function $e(v)$, we have  
	\begin{align*}
		&\quad\,\int^1_0\int_\Omega\Big(\int_{\R^3}\chi_0(1-\chi_1)(Z_k-Z)e(v)\,dv\Big)^2\,dxds\\
		&\lesssim \int^1_0\int_\Omega\int_{\R^3}(1-\chi_1)^2\big(|\chi_0Z_k|^2-|\chi_0Z|^2\big)e(v)\,dvdxds\\
		&\lesssim \Big(\int_{0\le s\le\eta_0}\int_{\Omega\times\R^3}+\int_{1-\eta_0\le s\le1}\int_{\Omega\times\R^3}+\int_0^1\int_{\Omega}\int_{|v|\ge \frac{1}{\eta_0}}\Big)\cdots\\
		&\lesssim \eta_0\sup_{0\le s\le 1}\int_{\Omega\times\R^3}\<v\>^{\gamma}\big(|\chi_0Z_k|^2-|\chi_0Z|^2\big)\,dvdx \lesssim \eta_0, 
	\end{align*}
for any $\eta_0>0$. 
Combining with the compactness of \eqref{410}, up to a subsequence, the macroscopic parts of $\chi_0Z_k$ satisfy $\chi_0\P Z_k\to \chi_0\P Z =\chi_0Z$ strongly in $L^2([0,1]\times\Omega\times\R^3)$. Therefore, in light of $\int^1_0\|(\I-\P)Z_k(s)\|^2_{L^2_xL^2_D}\,ds\to 0$ in \eqref{49b}, the remaining microscopic parts of $\chi_0Z_k$ satisfy $\lim_{k\to 0}\int^1_0\|\chi_0(\I-\P)Z_k(s)\|^2_{L^2_xL^2_D}\,ds=0$ and we conclude Lemma \ref{Lem44}. 	
\end{proof}

Next, as in \cite{Guo2009}, we will prove that there is no concentration at the boundary $\pa\Omega$ so that $Z_k\to Z$ strongly in $[0,1]\times\Omega\times\R^3$. Let 
\begin{align*}
	\Omega_{\ve^4} \equiv \{x\in\Omega: \xi(x)<-\ve^4\}. 
\end{align*}
To this end, we will establish a careful energy estimate near the boundary in the thin shell-like region $[0,1]\times\{\Omega\setminus\Omega_{\ve^4}\}\times\R^3$. Recall that  $n(x)=\frac{\na_x\xi}{|\na_x\xi|}\neq 0$ are well-defined and smooth 
on $\Omega\setminus\Omega_{\ve^4}$ for $\ve$ small. For $m>1/2$ and any $(x,v)$, we define the outward moving (inward moving) indicator function $\chi_+$ $(\chi_-)$ as 
\begin{align*}
	\chi_+(x,v) &= \1_{\Omega\setminus\Omega_{\ve^4}}(x)\1_{\{|v|\le \ve^{-m},v\cdot n(x)>\ve\}}(v),\\
	\chi_-(x,v) &= \1_{\Omega\setminus\Omega_{\ve^4}}(x)\1_{\{|v|\le \ve^{-m},v\cdot n(x)<-\ve\}}(v).
\end{align*}
The main strategy is to show that the moving (non-grazing) parts $\chi_\pm Z_k$ are controlled by the inner boundary values of $Z_k$ on $\pa\Omega_{\ve^4}=\{\xi(x)=-\ve^4\}$, which are further controlled by the (compact) interior parts of $Z_k$. Hence, no concentration is possible. 

\begin{Lem}\label{Lem45}
	\begin{align}\label{412a}
		\int^1_0\int_{\Omega\setminus\Omega_{\ve^4}}\int_{\substack{|v\cdot n|\le\ve\\\text{or }|v|\ge \ve^{-m}}}\<v\>^{\gamma}|Z_k(s,x,v)|^2\,dxdvds\le C\ve+\frac{C}{k}. 
	\end{align}
\end{Lem}
\begin{proof}
Let $\P Z_k= \{a_k(t,x)+v\cdot b_k(t,x)+|v|^2c_k(t,x)\}\sqrt\mu$. Since $\sup_{k\ge 1}\int^1_0\|Z_k\|^2_{L^2_xL^2_D}ds$ is finite and $[1,v,|v|^2]\sqrt\mu$ are linearly independent, there exists $C>0$ (independent of $k$) such that 
\begin{align}\label{411}
	\int^1_0\|[a_k,b_k,c_k]\|^2_{L^2_x}\,ds \le C\int^1_0\|Z_k(s)\|^2_{L^2_xL^2_D}\,ds\le C.
\end{align}
By \eqref{49}, we can split:
\begin{align*}
	&\quad\,\int^1_0\int_{\Omega\setminus\Omega_{\ve^4}}\int_{\substack{|v\cdot n|\le\ve\\\text{or }|v|\ge \ve^{-m}}}\<v\>^{\gamma}|Z_k(s,x,v)|^2\,dxdvds\\
	&\le 	\int^1_0\int_{\Omega\setminus\Omega_{\ve^4}}\int_{\substack{|v\cdot n|\le\ve\\\text{or }|v|\ge \ve^{-m}}}\<v\>^{\gamma}\big(|\P Z_k(s,x,v)|^2+|(\I-\P) Z_k(s,x,v)|^2\big)\,dxdvds\\
	&\le \int^1_0\int_{\Omega\setminus\Omega_{\ve^4}}\big|[a_k,b_k,c_k]\big|^2\Big\{\int_{\substack{|v\cdot n|\le\ve\\\text{or }|v|\ge \ve^{-m}}}\<v\>^{\gamma+2}\mu\,dv\Big\}\,dxds + \frac{C}{k}. 
\end{align*}
We note that the inner $v$-integral above is bounded uniformly in $x$. In fact, by a change of variable $v_{\parallel}=\{v\cdot n(x)\}n(x)$ and $v_\perp=v-v_{\parallel}$ for $|v\cdot n(x)|\le \ve$, the inner integral is bounded by the sum of 
\begin{align*}
	&\int_{|v\cdot n(x)|\le \ve} \<v\>^{\gamma+2}\mu\,dv\le C\int^\ve_{-\ve}dv_{\parallel}\int_{\R^2}e^{-|v_\perp|^2/8}\,dv_\perp\le C\ve\\
		&\text{ and }\int_{|v|\ge \ve^{-m}}\<v\>^{\gamma+2}\mu\,dv\le C\ve. 
\end{align*}
The Lemma \eqref{Lem45} thus follows from \eqref{411}. 
\end{proof}

To study the non-grazing parts $\chi_\pm Z_k$, we apply the same argument in \cite[eq. (99), pp. 745]{Guo2009} to obtain that 
\begin{align}\label{412}
	\int_{\Omega\setminus\Omega_{\ve^4}}\int_{\substack{|v\cdot n|\ge\ve\\|v|\le \ve^{-m}}}|Z_k(s,x,v)|^2dxdv\le C\ve. 
\end{align}

\begin{proof}
	[Proof of Lemma \ref{Lem42}]
	We are now ready to prove compactness of $Z_k$. We split 
	\begin{multline}\label{413}
		\int^1_0\int_{\Omega}\int_{\R^3}\<v\>^{\gamma}|Z_k(s,x,v)-Z(s,x,v)|^2dsdxdv\\ = 
		2\int^1_0\int_{\Omega\setminus\Omega_{\ve^4}}\int_{\substack{|v\cdot n|\le\ve\\\text{or }|v|\ge \ve^{-m}}}\<v\>^{\gamma}|Z_k(s,x,v)|^2dsdxdv 
		+ 2\int^1_0\int_{\Omega\setminus\Omega_{\ve^4}}\int_{\substack{|v\cdot n|\ge\ve\\|v|\le \ve^{-m}}}\<v\>^{\gamma}|Z_k(s,x,v)|^2dsdxdv\\
		+ 2\int^1_0\int_{\Omega\setminus\Omega_{\ve^4}}\int_{\R^3}\<v\>^{\gamma}|Z(s,x,v)|^2dsdxdv+ \int^1_0\int_{\Omega_{\ve^4}}\int_{\R^3}\<v\>^{\gamma}|Z_k(s,x,v)-Z(s,x,v)|^2dsdxdv. 
	\end{multline}
Clearly, using \eqref{412a}, the first right-hand term of \eqref{413} is bounded by $C\ve+\frac{C}{k}$. The second term on the right-hand of \eqref{413} is bounded by $C\ve$ according to \eqref{412} and $\gamma<0$. By Lemma \ref{Lem43}, the third right-hand term of \eqref{413} is bounded by 
\begin{align*}
	 C|\Omega\setminus\Omega_{\ve^4}|\le C\ve,
\end{align*}
where $C$ depends on $a_0,c_0,c_1,c_2,b_0,b_1$ and $\varpi$. Here $|\Omega\setminus\Omega_{\ve^4}|$ means the Lebesgue measure of set $\Omega\setminus\Omega_{\ve^4}$. The last term of \eqref{413} tends to zero as $k\to\infty$ by \eqref{412b} and $\gamma<0$. We hence deduce the strong convergence 
\begin{align}\label{414}
	\limsup_{k\to \infty}\int^1_0\int_{\Omega}\int_{\R^3}\<v\>^{\gamma}|Z_k(s,x,v)-Z(s,x,v)|^2dsdxdv = 0,
\end{align}
by letting $\ve\to 0$. From our normalization \eqref{48}, we know that 
\begin{equation}\label{414a}
	\int^1_0\|Z(s)\|^2_{L^2_xL^2_D}ds = \lim_{k\to\infty}\int^1_0\|Z_k(s)\|^2_{L^2_xL^2_D}=1. 
\end{equation}

Next, we study the boundary conditions which $Z$ satisfies. In fact, $Z_k$ and $Z$ satisfy 
\begin{align*}
	\{\pa_t+v\cdot\na_x\}\Big\{\1_{\{|v|\le \ve^{-m}\}}(Z_k-Z)\Big\}=-\1_{\{|v|\le \ve^{-m}\}}L\{\I-\P\}Z_k. 
\end{align*}
Using Ukai's trace theorem \cite[Theorem 5.1.1]{Ukai1986}, \eqref{414} and \eqref{49b}, for any fixed $\ve>0$, we have 
\begin{align*}
	&\quad\,\lim_{k\to\infty}\int^1_0\big\|\1_{\{|v\cdot n(x)|\ge \frac{\ve}{2},|v|\le \ve^{-m}\}}(Z_k-Z)\big\|^2_{\gamma}ds\\
	&\le C\lim_{k\to\infty}\Big[\int^1_0\big\|\1_{\{|v|\le \ve^{-m}\}}(Z_k-Z)\big\|^2_{L^2_xL^2_v}ds
	+ \int^1_0\big\|\1_{\{|v|\le \ve^{-m}\}}L\{\I-\P\}Z_k(s)\big\|^2_{L^2_xL^2_v}ds\Big]\\
	&\le C\lim_{k\to\infty}\int^1_0\big\|\{\I-\P\}Z_k(s)\big\|^2_{L^2_xL^2_D}ds = 0.
\end{align*}
Thus, for the inflow boundary condition, by \eqref{49}, 
\begin{align*}
	\int^1_0\big\|\1_{\{|v\cdot n(x)|\ge \frac{\ve}{2},|v|\le \ve^{-m}\}}Z(s)\big\|^2_{\gamma}ds = \lim_{k\to\infty}\int^1_0\big\|\1_{\{|v\cdot n(x)|\ge \frac{\ve}{2},|v|\le \ve^{-m}\}}Z_k(s)\big\|^2_{\gamma}ds= 0.
\end{align*}
This is valid for $\ve>0$. Thus, $Z\equiv 0$ on $\gamma$. 

Finally, as in \cite[pp. 747]{Guo2009}, for any $t$ and $x\in\pa\Omega$, and $v\in\R^3$, by comparing the coefficients in front of polynomials of $v$ in \eqref{499} and $Z=0$ on $\gamma$. We deduce that 
\begin{equation*}
	\Big\{\frac{c_0}{2}|x|^2-b_0\cdot x+a_0\Big\}=\big\{-c_0tx-c_1x+\varphi\times x+b_0t+b_1\big\}\times v=\Big\{\frac{c_0t^2}{2}+c_1t+c_2\Big\}|v|^2=0. 
\end{equation*}
Therefore, $c_0=c_1=c_2=0$, and $b_0=0$. Then $a_0=0$ and $\varpi\times x+b_1\equiv 0$, or 
\begin{align*}
	\varpi^2x_3-\varpi^3x_2+b^1_1=-\varpi^1x_3+\varpi^3x_1+b^2_1=\varpi^1x_2-\varpi^2x_1+b^3_1\equiv 0,
\end{align*}
for all $x\in\pa\Omega$. Notice that $\xi(x)=0$ is two dimensional, so we may assume that $(x_1,x_2)$ are (locally) independent. Hence $\varpi^1=\varpi^2=b^3_1=0$ then $\varpi^3=b^2_1=0$, and finally $b^1_1=0$. Therefore, we deduce $Z\equiv 0$ and this contradicts with \eqref{414a}, and we conclude Lemma \ref{Lem42}. 
\end{proof}

\begin{Thm}\label{Thm41}
	Let $f(t,x,v)\in L^2_{x,v}$ be the (unique) solution to the cutoff Boltzmann equation \eqref{1a} with inflow boundary condition \eqref{2} and trace $f|_{\gamma}\in L^2_{\text{loc}}(\R_+;L^2(\gamma))$. Then there are $\lam>0$ and $C>0$ such that for all $0\le t\le \infty$, 
	\begin{align*}
		e^{2\lam t}\|f(t)\|^2_{L^2_{x,v}} \le 4\Big\{\|f_0\|_{L^2_{x,v}}-\int^t_0e^{2\lam s}\int_{v\cdot n<0}\int_{\R^3}v\cdot n|f(s)|^2dvdS(x)ds\Big\}. 
	\end{align*}
\end{Thm}


\begin{proof}
	[Proof of Theorem \ref{Thm41}]
	For any solution $f$ to the linearized Boltzmann equation \eqref{1a}, $e^{\lam t}Wf(t)$ satisfies 
	\begin{align}\label{43}
		\pa_t(e^{\lam t}Wf) + v\cdot \na_x (e^{\lam t}Wf) + q|v|^2\<v\>^{-1}Wf - WL(e^{\lam t}f) - \lam e^{\lam t}Wf = 0.
	\end{align}
For any $t>0$, let $0\le N\le t\le N+1$ with $N$ being an integer. We split $[0,t] = [0,N]\cup[N,t]$. For the time interval $[N,t]$, we let $\lam=0$ and take the inner product of \eqref{43} with $f$ over $[N,t]\times\Omega\times\R^3$ to obtain 
\begin{multline}\label{45a}
	\|f\|^2_{L^2_{x,v}} + 2c_1\int^t_N\|\{\I-\P\}f\|^2_{L^2_{x}L^2_D}\,ds + \int^t_N\int_{v\cdot n>0}\int_{\R^3}v\cdot n|f(s)|^2dvdS(x)ds \\\le  \|f(N)\|^2_{L^2_{x,v}} - \int^t_N\int_{v\cdot n<0}\int_{\R^3}v\cdot n|f(s)|^2dvdS(x)ds, 
\end{multline}when $q=0$. 
Here we apply \eqref{e21} and \eqref{e22}. 
For the time interval $[0,N]$ (we may assume $N\ge 1)$, we can take the inner product of \eqref{43} with $e^{\lam t}Wf$ over $[0,N]\times\Omega\times\R^3$ to obtain 
\begin{multline}\label{45c}
	e^{2\lam N}\|f(N)\|^2_{L^2_{x,v}} + 2c_1\int^N_0e^{2\lam s}\|\{\I-\P\}f\|^2_{L^2_{x}L^2_D}\,ds - \lam \int^N_0e^{2\lam s}\|f(s)\|_{L^2_{x,v}}^2\,ds\\ 
	+ \int^N_0\int_{v\cdot n>0}\int_{\R^3}v\cdot ne^{2\lam s}|f(s)|^2dvdS(x)ds\\ \le  \|f(0)\|^2_{L^2_{x,v}} - \int^N_0\int_{v\cdot n<0}\int_{\R^3}v\cdot ne^{2\lam s}|f(s)|^2dvdS(x)ds, 
\end{multline}
when $q=0$, and 
\begin{multline}\label{45d}
	e^{2\lam N}\|Wf(N)\|^2_{L^2_{x,v}} + 2q\int^N_0e^{2\lam s}\||v|\<v\>^{-\frac{1}{2}}Wf\|^2_{L^2_{x}L^2_D}\,ds - \lam \int^N_0e^{2\lam s}\|Wf(s)\|_{L^2_{x,v}}^2\,ds\\ 
	+ \int^N_0\int_{v\cdot n>0}\int_{\R^3}v\cdot ne^{2\lam s}|Wf(s)|^2dvdS(x)ds\\ \le C_1\int^N_0e^{2\lam s}\|f\|^2_{L^2_{x}L^2_D}\,ds + \|Wf(0)\|^2_{L^2_{x,v}} - \int^N_0\int_{v\cdot n<0}\int_{\R^3}v\cdot ne^{2\lam s}|Wf(s)|^2dvdS(x)ds, 
\end{multline}
when $q>0$. 
Dividing the time interval into $\cup^{N-1}_{k=0}[k,k+1)$ and letting $f_k(s,x,v)\equiv f(k+s,x,v)$ for $k=0,1,\dots,N-1$, we deduce from \eqref{45c} and \eqref{45d} that 
\begin{multline}\label{45e}
	e^{2\lam N}\|f(N)\|^2_{L^2_{x,v}} + \sum^{N-1}_{k=0}\int^1_0\Big(2c_1e^{2\lam (k+s)}\|\{\I-\P\}f_k(s)\|^2_{L^2_{x}L^2_D}\,ds - \lam e^{2\lam (k+s)}\|f_k(s)\|_{L^2_{x,v}}^2\Big)\,ds \\
	+ \sum^{N-1}_{k=0}\int^1_0\int_{v\cdot n>0}\int_{\R^3}v\cdot ne^{2\lam (k+s)}|f_k(s)|^2dvdS(x)ds \\\le  \|f(0)\|^2_{L^2_{x,v}} - \sum^{N-1}_{k=0}\int^1_0\int_{v\cdot n<0}\int_{\R^3}v\cdot ne^{2\lam (k+s)}|f_k(s)|^2dvdS(x)ds, 
 \end{multline}
when $q=0$, and 
\begin{multline}\label{45f}
e^{2\lam N}\|Wf(N)\|^2_{L^2_{x,v}} + \sum^{N-1}_{k=0}\int^1_0\Big(2qe^{2\lam (k+s)}\||v|\<v\>^{-\frac{1}{2}}Wf_k(s)\|^2_{L^2_{x}L^2_v}\,ds - \lam e^{2\lam (k+s)}\|Wf_k(s)\|_{L^2_{x,v}}^2\Big)\,ds \\
+ \sum^{N-1}_{k=0}\int^1_0\int_{v\cdot n>0}\int_{\R^3}v\cdot ne^{2\lam (k+s)}|Wf_k(s)|^2dvdS(x)ds \\\le 
C_1\sum_{k=0}^{N-1}\int^1_0e^{2\lam (k+s)}\|f_k(s)\|^2_{L^2_{x}L^2_D}\,ds
 + \|Wf(0)\|^2_{L^2_{x,v}}\\ - \sum^{N-1}_{k=0}\int^1_0\int_{v\cdot n<0}\int_{\R^3}v\cdot ne^{2\lam (k+s)}|Wf_k(s)|^2dvdS(x)ds, 
\end{multline}
when $q>0$. 
Let $\delta_0>0$ be chosen later. Multiplying $\delta_0e^{2\lam k}$ with \eqref{44} to each $f_k(s,x,v)$ and then summing up over $k$ yields 
\begin{multline}\label{46}
\frac{\delta_0}{M}\sum^{N-1}_{k=0}e^{2\lam k}\int^1_0\|\P f_k(s)\|_{L^2_xL^2_D}^2\,ds\le \delta_0\sum^{N-1}_{k=0}e^{2\lam k}\Big\{c_1\int^1_0\|(\I-\P) f_k(s)\|_{L^2_xL^2_D}^2\,ds \\ + \int^1_0\int_{\pa\Omega}\int_{\R^3}|v\cdot n||f(s)|^2dvdS(x)ds\Big\}. 
\end{multline}
Substituting \eqref{46} into \eqref{45e} and choosing $\delta_0<\min\{1,Mc_1\}$, we deduce that 
\begin{multline}\label{45g}
	e^{2\lam N}\|f(N)\|^2_{L^2_{x,v}} + \sum^{N-1}_{k=0}\frac{\delta_0e^{2\lam k}}{M}\int^1_0\| f_k(s)\|_{L^2_xL^2_D}^2\,ds- \lam \sum^{N-1}_{k=0}\int^1_0e^{2\lam (k+s)}\|f_k(s)\|_{L^2_{x,v}}^2\,ds\\
	 + \sum^{N-1}_{k=0}\Big\{
	 (1-\delta_0)\int^1_0\int_{v\cdot n>0}\int_{\R^3}v\cdot ne^{2\lam (k+s)}|f_k(s)|^2dvdS(x)ds
	  \\
	\le  \|f(0)\|^2_{L^2_{x,v}} - (1+\delta_0)\sum^{N-1}_{k=0}\int^1_0\int_{v\cdot n<0}\int_{\R^3}v\cdot ne^{2\lam (k+s)}|f_k(s)|^2dvdS(x)ds.
\end{multline}
Then we apply \eqref{45g} to eliminate the first right-hand term of \eqref{45f}. 
Taking the linear combination $\kappa\times\eqref{45f}+\eqref{45g}$ with $\kappa>0$ sufficiently small satisfying $\kappa <\frac{\delta_0}{2MC_1e^{2\lam}}$ and $\kappa W(x,v)\le 1$ for any $x$ and $v$, one has 
\begin{multline}\label{45i}
	e^{2\lam N}\|f(N)\|^2_{L^2_{x,v}} + \sum^{N-1}_{k=0}\Big\{\frac{\delta_0e^{2\lam k}}{2M}\int^1_0\| f_k(s)\|_{L^2_xL^2_D}^2\,ds
	+2\kappa q\int^1_0e^{2\lam (k+s)}\||v|\<v\>^{-\frac{1}{2}}Wf_k(s)\|^2_{L^2_{x}L^2_v}\,ds\Big\}\\ -  \sum^{N-1}_{k=0}\int^1_02e^{2\lam (k+s)}\lam\|f_k(s)\|_{L^2_{x,v}}^2\,ds 
	\\
	+ 	(1-\delta_0)\sum^{N-1}_{k=0}
\int^1_0\int_{v\cdot n>0}\int_{\R^3}v\cdot ne^{2\lam (k+s)}|f_k(s)|^2dvdS(x)ds
	\\
	\le  2\|f(0)\|^2_{L^2_{x,v}} - 2(1+\delta_0)\sum^{N-1}_{k=0}\int^1_0\int_{v\cdot n<0}\int_{\R^3}v\cdot ne^{2\lam (k+s)}|f_k(s)|^2 dvdS(x)ds.
\end{multline}
Here we can choose $\lam>0$ small enough to control the negative term on the left-hand side of \eqref{45i}. 
Noticing $\|f_k\|_{L^2_{x,v}}\le C_1\big(\|f_k\|_{L^2_xL^2_D}+\||v|\<v\>^{-\frac{1}{2}}Wf_k\|_{L^2_xL^2_v}\big)$ for some $C_1$, we choose $\lam>0$ in \eqref{45i} sufficiently small such that $2C_1e^{2\lam}\lam \le \min\{\frac{\delta_0e^{2\lam k}}{4M},\kappa q\}$ so as to obtain 
\begin{multline}\label{45h}
	e^{2\lam N}\|f(N)\|^2_{L^2_{x,v}} + \sum^{N-1}_{k=0}\Big\{\frac{\delta_0e^{2\lam k}}{4M}\int^1_0\| f_k(s)\|_{L^2_xL^2_D}^2\,ds
	+\kappa q\int^1_0e^{2\lam (k+s)}\||v|\<v\>^{-\frac{1}{2}}Wf_k(s)\|^2_{L^2_{x}L^2_v}\,ds\Big\}
	\\
	+ 	(1-\delta_0)\sum^{N-1}_{k=0}
	\int^1_0\int_{v\cdot n>0}\int_{\R^3}v\cdot ne^{2\lam (k+s)}|f_k(s)|^2dvdS(x)ds
	\\
	\le  2\|f(0)\|^2_{L^2_{x,v}} - 2(1+\delta_0)\sum^{N-1}_{k=0}\int^1_0\int_{v\cdot n<0}\int_{\R^3}v\cdot ne^{2\lam (k+s)}|f_k(s)|^2 dvdS(x)ds.
\end{multline}
Notice that $e^{2\lam t}=e^{2\lam (t-N)}e^{2\lam s}$ for $s\ge N$, and since $t\le N+1$, we can choose  $\lam>0$ small such that $e^{2\lam (t-N)}(1+\delta_0)\le 2$. Hence, multiplying $e^{2\lam t}$ with \eqref{45a} and combining it with \eqref{45h} yields 
\begin{align*}
&\quad\,e^{2\lam t}\|f\|^2_{L^2_{x,v}}  + e^{2\lam t}\int^t_N\int_{v\cdot n>0}\int_{\R^3}v\cdot n|f(s)|^2dvdS(x)ds \\
&\le  e^{2\lam (t-N)}e^{2\lam N}\|f(N)\|^2_{L^2_{x,v}} - e^{2\lam t}\int^t_N\int_{v\cdot n<0}\int_{\R^3}v\cdot n|f(s)|^2dvdS(x)ds\\
&\le 4\|f(0)\|^2_{L^2_{x,v}} - 4\sum^{N-1}_{k=0}\int^1_0\int_{v\cdot n<0}\int_{\R^3}v\cdot ne^{2\lam (k+s)}|f_k(s)|^2 dvdS(x)ds\\
&\qquad\qquad\qquad\qquad\qquad - e^{2\lam (t-N)}\int^t_N\int_{v\cdot n<0}\int_{\R^3}v\cdot ne^{2\lam s}|f(s)|^2dvdS(x)ds\\ 
&\le 4\|f(0)\|^2_{L^2_{x,v}} - 4\int^t_0\int_{v\cdot n<0}\int_{\R^3}v\cdot ne^{2\lam s}|f(s)|^2 dvdS(x)ds. 
\end{align*}
This completes the proof of Theorem \ref{Thm41}. 
\end{proof}

Let $w=(1+\rho^2|v|^2)^\beta$ with $\rho>0$ and $\beta\in\R$. 
Here we list the theory of $L^\infty$ decay for the weighted $h=wWf$ $(k\ge 0)$ of the linear cutoff Boltzmann equation \eqref{1} with the inflow boundary condition \eqref{2} satisfing \eqref{prioricutoff}. We consider the homogeneous transport equation 
\begin{align}\label{51}
	\{\pa_t+v\cdot\na_x+\nu\}f=0, \quad f(0,x,v)=f_0(x,v),\quad f|_{\gamma_-}=g,
\end{align}
with the inflow datum $g$ satisfying \eqref{prioricutoff}.
By \eqref{W1}, we can rewrite \eqref{51} in the form of $h=wWf$ $(k\ge 0)$ as 
\begin{align}\label{52}
	\{\pa_t+v\cdot\na_x+q{|v|^2}\<v\>^{-1}+\nu\}h=0, \quad h(0,x,v)=wWf_0(x,v),\quad h|_{\gamma_-}=wWg. 
\end{align} 
Fix $q>0$. Since $\nu(v)\approx \<v\>^{\gamma}$, there exist $\nu_0,\nu_1>0$ such that 
\begin{align*}
	\nu_0\<v\>\le q{|v|^2}\<v\>^{-1}+\nu(v)\le \nu_1\<v\>.
\end{align*} Then we can regard  $q{|v|^2}\<v\>^{-1}+\nu$ as a whole dissipation term and its behavior is quite similar to $\nu(v)$ in the hard potential case. In view of this, we can apply similar arguments in \cite[Lemma 12, Lemma 13 and Theorem 6]{Guo2009} to obtain the following lemmas and theorem. 

\begin{Lem}\label{Lem51}
	Let $f_0\in L^\infty(\Omega\times\R^3)$ and $Wg\in L^\infty(\gamma_-\times\R^3)$. Let $G(t)$ be the solution to equation \eqref{52}
	\begin{align*}
		\{\pa_t+v\cdot\na_x+q{|v|^2}\<v\>^{-1}+\nu\}G(t)h_0=0, \quad G(0)h_0=h_0(x,v),\quad \{G(t)h\}|_{\gamma_-}=wWg.
	\end{align*}
	For $(x,v)\notin\gamma_0\cup \gamma_-$, 
	\begin{align*}
		\{G(t)h_0\}(t,x,v)&=\1_{t-t_{\b}\le 0}e^{-(q{|v|^2}\<v\>^{-1}+\nu(v))t}h_0(x-tv,v)\\
		&\qquad+\1_{t-t_\b>0}e^{-(q{|v|^2}\<v\>^{-1}+\nu(v))t_\b}\{wWg\}(t-t_\b,x-t_\b v,v),
	\end{align*}
where $t_\b(x,v)$ is the backward exit time defined in \eqref{backward}. Moreover, 
\begin{align*}
	\sup_{t\ge 0}e^{\nu_0t}\|G(t)h_0\|_{L^\infty(\Omega\times\R^3)}\le \|h_0\|_{L^\infty_{x,v}}+\sup_{s\ge 0}e^{\nu_0s}\|wWg(s)\|_{L^\infty(\gamma_-\times\R^3)}.
\end{align*}

\end{Lem}
%

\begin{Lem}\label{Lem52}
	Lem $\Omega$ be strictly convex as in \eqref{convex}. Let $h_0(x,v)$ be continuous in $\overline\Omega\times\R^3\setminus\gamma_0$, $g$ be continuous in $[0,\infty)\times\{\pa\Omega\times\R^3\setminus\gamma_0\}$, $p(t,x,v)$ be continuous in the interior of $[0,\infty)\times\pa\Omega\times\R^3$ and $\sup_{[0,\infty)\times\pa\Omega\times\R^3}|\frac{p(t,x,v)}{q{|v|^2}\<v\>^{-1}+\nu(v)}|<\infty$. Let $h(t,x,v)$ be the solution to 
	\begin{align*}
		\{\pa_t+v\cdot\na_x+q{|v|^2}\<v\>^{-1}+\nu\}h=p,\quad h(0)=h_0,\quad h|_{\gamma_-}=wWg. 
	\end{align*}
We further assume the compatibility condition on $\gamma_-$:
\begin{align*}
	h_0(x,v) = \{wWg\}(0,x,v). 
\end{align*}
Then $h(t,x,v)$ is continuous on $[0,\infty)\times\{\overline\Omega\times\R^3\setminus\gamma_0\}$ and can be written as 
\begin{align*}
	h(t,x,v)=e^{-(q{|v|^2}\<v\>^{-1}+\nu(v))t}h_0(x-tv,v) + \int^t_0e^{-(q{|v|^2}\<v\>^{-1}+\nu(v))(t-s)}p(s,x-(t-s)v,v)\,ds
\end{align*}
if $t\le t_\b$ and 
\begin{multline*}
	h(t,x,v)=e^{-(q{|v|^2}\<v\>^{-1}+\nu(v))t_\b}\{wWg\}(t-t_\b,x_\b,v)\\+\int^t_{t-t_\b}e^{-(q{|v|^2}\<v\>^{-1}+\nu(v))(t-s)}p(s,x-(t-s)v,v)\,ds
\end{multline*}
if $t>t_\b$. 

\end{Lem}

\begin{Thm}\label{Thm51}
	Assume $\<v\>w^{-2}\in L^1(\R^3_v)$. Let $U(t)h_0$ be the solution to the weighted linearized cutoff Boltzmann equation 
	\begin{align*}
		\{\pa_t+v\cdot\na_x+q{|v|^2}\<v\>^{-1}+\nu-K_{w,W}\}U(t)h_0=0,\quad U(0)h_0=h_0,\quad \{U(t)h_0\}|_{\gamma_-}=wWg,
	\end{align*}
where $K_{w,W}$ is given by $K_{w,W}h=wWK(\frac{h}{wW})$. Then there exists $0<\lam<\lam_0$ such that 
\begin{align}\label{54}
	\sup_{t\ge 0}e^{\lam t}\|U(t)h_0\|_{L^\infty_{x,v}}\le C\big\{\|h_0\|_{L^\infty_{x,v}}+\sup_{s\ge 0}e^{\lam_0s}\|wWg(s)\|_{L^\infty(\gamma_-\times\R^3)}\big\},
\end{align}
where $\lam_0>0$ is given in \eqref{prioricutoff}. 
\end{Thm}

Notice that in the proof of Theorem \ref{Thm51}, we crucially apply Theorem \ref{Thm41} and Lemmas \ref{Lem51} and \ref{Lem52}. Noticing $W\approx 1$, the behavior of $K_{w,W}$ is similar to $K_{w,1}$ and the similar arguments in \cite[Theorem 6, pp.752]{Guo2009} can be applied. 
Now we can prove the exponential decay for cutoff Boltzmann equation.

\begin{proof}
	[Proof of Theorem \ref{Maincutoff}]
	In order to prove the existence of equation \eqref{cutoff}, we shall apply iteration $\{h^n(t,x,v)\}$:
	\begin{align}\label{55}
		\{\pa_t+v\cdot\na_x+q{|v|^2}\<v\>^{-1}+\nu-K_{w,W}\}h^{n+1}=wW\Gamma\Big(\frac{h^n}{wW},\frac{h^n}{wW}\Big),
	\end{align}
	with $h^0=0$ and $h^{n+1}|_{t=0}=h_0$, $h^{n+1}|_{\gamma_-}=wWg$. 
	In order to solve \eqref{55}, we split $h^{n+1}=h^{n+1}_g+h^{n+1}_\Gamma$ for $n\ge 0$, where $h^{n+1}_g$ solves the homogeneous linear weighted cutoff Boltzmann equation 
	\begin{align}\label{56}
		\{\pa_t+v\cdot\na_x+q{|v|^2}\<v\>^{-1}+\nu-K_{w,W}\}h^{n+1}_g = 0, \\
		\notag h^{n+1}_g|_{\gamma_-}=wWg, \quad h^{n+1}_g|_{t=0}=h_0,	
	\end{align}
while $h^{n+1}_\Gamma$ satisfies the inhomogeneous weighted Boltzmann equation with {\it zero} boundary and initial values:
\begin{align}\label{56a}
	\{\pa_t+v\cdot\na_x+q{|v|^2}\<v\>^{-1}+\nu-K_{w,W}\}h^{n+1}_\Gamma &= wW\Gamma\Big(\frac{h^n}{wW},\frac{h^n}{wW}\Big),\\
	\notag \quad h^{n+1}_\Gamma|_{\gamma_-}=0,\  h^{n+1}_\Gamma|_{t=0}&=0.
\end{align}
	Applying Theorem \ref{Thm51} to \eqref{56}, there exists $0<\lam<\lam_0$ such that 
	\begin{align}\label{57}
		\sup_{t\ge 0}e^{\lam t}\|h^{n+1}_g(t)\|_{L^\infty_{x,v}}\le C\big\{\|h_0\|_{L^\infty_{x,v}}+\sup_{s\ge 0}e^{\lam_0s}\|wWg(s)\|_{L^\infty(\gamma_-\times\R^3)}\big\}.
	\end{align}
	To find the estimate on $h^{n+1}_\Gamma$, we denote $U(t)$ to be the solution operator for the linear weighted Boltzmann equation 
	\begin{align*}
		\{\pa_t+v\cdot\na_x+q{|v|^2}\<v\>^{-1}+\nu-K_{w,W}\}h =0, \quad h|_{\gamma_-}=0,\ h|_{t=0}=0.
	\end{align*}
Then by Duhamel's principle, the solution to \eqref{56a} can be written as 
\begin{align}\label{57a}
	h^{n+1}_\Gamma = \int^t_0U(t-s)wW\Gamma\Big(\frac{h^n(s)}{wW},\frac{h^n(s)}{wW}\Big)\,ds. 
\end{align}
Recall the basic estimate on $\Gamma$ from \cite[Lemma 5, pp. 730]{Guo2009}:
\begin{align}\label{58}
	|w\Gamma(g_1,g_2)|\le C\<v\>^\gamma |wg_1|_{L^\infty_v}|wg_2|_{L^\infty_v}.
\end{align}
Noticing $W\approx 1$, we apply the energy estimate \eqref{54} to \eqref{57a} and deduce that  
\begin{align}\label{57b}\notag
	\|h^{n+1}_\Gamma\|_{L^\infty_{x,v}}&\le C\int^t_0e^{-\lam(t-s)}\Big\|wW\Gamma\Big(\frac{h^n(s)}{wW},\frac{h^n(s)}{wW}\Big)\Big\|_{\infty_{x,v}}ds\\
	&\notag\le C\int^t_0e^{-\lam(t-s)}e^{-2\lam s}\,ds\ \Big\{\sup_{s\ge 0}e^{\lam s}\|h^n(s)\|_{L^\infty_{x,v}}\Big\}^2\\
	&\le Ce^{-\lam t}\Big\{\sup_{s\ge 0}e^{\lam s}\|h^n(s)\|_{L^\infty_{x,v}}\Big\}^2.
\end{align}
In view of \eqref{57} and \eqref{57b}, by induction, there exists $\delta>0$ small enough such that if \eqref{prioricutoff} is valid, then 
\begin{align*}
	\sup_{n\ge 1}\sup_{t\ge 0}e^{\lam t}\|h^{n+1}_g(t)\|_{L^\infty_{x,v}}\le C\big\{\|h_0\|_{L^\infty_{x,v}}+\sup_{s\ge 0}e^{\lam_0s}\|wWg(s)\|_{L^\infty(\gamma_-\times\R^3)}\big\}. 
\end{align*}
To find the limit of $h^n$, we subtract $h^{n+1}-h^n$ to obtain 
\begin{align*}
	&\{\pa_t+v\cdot\na_x+q{|v|^2}\<v\>^{-1}+\nu-K_{w,W}\}\{h^{n+1}-h^n\}\\&\qquad\qquad\qquad\qquad=wW\Gamma\Big(\frac{h^n-h^{n-1}}{wW},\frac{h^n}{wW}\Big)+wW\Gamma\Big(\frac{h^{n-1}}{wW},\frac{h^n-h^{n-1}}{wW}\Big),\\
	&\{h^{n+1}-h^n\}|_{t=0}=0, \quad \{h^{n+1}-h^n\}|_{\gamma_-}=0. 
\end{align*}
Similar to \eqref{57b}, we deduce that 
\begin{align}\label{58b}\notag
	&\quad\,\|h^{n+1}-h^n\|_{L^\infty_{x,v}}\\\notag&\le \Big\|\int^t_0U(t-s)\Big(wW\Gamma\Big(\frac{h^n-h^{n-1}}{wW},\frac{h^n}{wW}\Big)(s)+wW\Gamma\Big(\frac{h^{n-1}}{wW},\frac{h^n-h^{n-1}}{wW}\Big)(s)\Big)\Big\|_{L^\infty_{x,v}}\\
	&\le Ce^{-\lam t}\sup_{s\ge 0}\big\{e^{\lam s}\|h^n(s)\|_{L^\infty_{x,v}}+e^{\lam s}\|h^{n-1}(s)\|_{L^\infty_{x,v}}\big\} \sup_{s\ge 0}e^{\lam s}\|h^n(s)-h^{n-1}(s)\|_{L^\infty_{x,v}}. 
\end{align}
Hence, by induction, $\{e^{\lam t}h^n\}$ is a Cauchy sequence, and the limit $h$ is the solution to cutoff Boltzmann equation \eqref{cutoff}. Then $f=\frac{h}{wW}$ is the desired solution to \eqref{1} with inflow boundary condition \eqref{2}. 

\medskip 

To prove the uniqueness of $h$ (and hence $f$), we assume that $\wt h$ is another solution to the full Boltzmann equation \eqref{cutoff} and $\sup_{0\le t\le T_0}\|\wt h(t)\|_{L^\infty_{x,v}}$ is sufficiently small. Then $h-\wt h$ satisfies 
\begin{align*}
	&\{\pa_t+v\cdot\na_x+q{|v|^2}\<v\>^{-1}+\nu-K_{w,W}\}\{h-\wt h\}\\&\qquad\qquad\qquad\qquad 
	= wW\Gamma\Big(\frac{h-\wt h}{wW},\frac{h}{wW}\Big)+wW\Gamma\Big(\frac{\wt h}{wW},\frac{h-\wt h}{wW}\Big),\\
	&\{h-\wt h\}|_{t=0}=0, \quad \{h-\wt h\}|_{\gamma_-}=0.
\end{align*}
Similar to \eqref{58b}, for $0\le t\le T$, we have 
\begin{align*}
	\|h-\wt h\|_{L^\infty_{x,v}}&\le C\Big\|\int^t_0U(t-s)\Big(wW\Gamma\Big(\frac{h-\wt h}{wW},\frac{h}{wW}\Big)(s)+wW\Gamma\Big(\frac{\wt h}{wW},\frac{h-\wt h}{wW}\Big)\Big)(s)\Big\|_{L^\infty_{x,v}}\\
	&\le C_T\sup_{s\ge 0}\big\{\|h(s)\|_{L^\infty_{x,v}}+e^{\lam s}\|\wt h(s)\|_{L^\infty_{x,v}}\big\} \sup_{s\ge 0}\|h(s)-\wt h(s)\|_{L^\infty_{x,v}}. 
\end{align*}
This implies that $\sup_{0\le t\le T}\|h-\wt h\|_{L^\infty_{x,v}}=0$ whenever $\sup_{0\le t\le T_0}\|h(t)\|_{L^\infty_{x,v}}$ and $\sup_{0\le t\le T_0}\|\wt h(t)\|_{L^\infty_{x,v}}$ are sufficiently small. This shows the uniqueness for \eqref{cutoff}.

\medskip 
Finally, if $\Omega$ is strictly convex, we claim that $h^{n+1}$ is continuous in $[0,\infty)\times\{\overline\Omega\times \R^3\setminus\gamma_0\}$ as in \cite[pp. 801]{Guo2009}. For a fixed $n\ge 1$, we use another iteration to solve $h^{n+1}$ as the limit of $n'\to\infty$:
\begin{align*}
	\{\pa_t+v\cdot\na_x+q{|v|^2}\<v\>^{-1}+\nu\}h^{n+1,n'+1}=K_{w,W}h^{n+1,n'}+wW\Gamma\Big(\frac{h^n}{wW},\frac{h^n}{wW}\Big),
\end{align*}
	with $h^{n+1,0}=0$ and $h^{n+1,n'+1}|_{t=0}=h_0$, $h^{n+1,n'+1}|_{\gamma_-}=wWg$. By induction over $n'$, we assume that $h^{n+1,n'}$ and $h^n$ are continuous in $[0,\infty)\times\{\overline\Omega\times \R^3\setminus\gamma_0\}$. Then it is straightforward to verify that $K_{w,W}h^{n+1,n'}$ and $\nu^{-1}wW\Gamma\Big(\frac{h^n}{wW},\frac{h^n}{wW}\Big)$ are continuous in the interior $[0,\infty)\times\Omega\times\R^3$; see \cite[pp. 802]{Guo2009}. 
	Moreover, we have from \eqref{58} that $\sup_{[0,\infty)\times\Omega\times\R^3}\big|\nu^{-1}wW\Gamma\Big(\frac{h^n}{wW},\frac{h^n}{wW}\Big)\big|<\infty$. Then by Lemma \ref{Lem52}, we deduce that $h^{n+1,n'+1}$ is continuous in $[0,\infty)\times\{\overline\Omega\times\R^3\setminus\gamma_0\}$. Finally, the difference $h^{n+1,n'+1}-h^{n+1,n'}$ satisfies 
	\begin{align*}
		\{\pa_t+v\cdot\na_x+q{|v|^2}\<v\>^{-1}+\nu\}\{h^{n+1,n'+1}-h^{n+1,n'}\}=K_{w,W}\{h^{n+1,n'}-h^{n+1,n'-1}\}, 
	\end{align*}
with {\it zero} boundary and initial conditions. Then by the boundedness of $K_{w,W}$ (for instance, \eqref{nu2}) and Duhamel's principle, we have 
\begin{align*}
	\sup_{0\le t\le T}\|h^{n+1,n'+1}-h^{n+1,n'}\|_{L^\infty_{x,v}}\le C_K\int^T_0\|h^{n+1,n'}-h^{n+1,n'-1}\|_{L^\infty_{x,v}}\,ds\le \cdots\le \frac{(C_KT)^{n'}}{n'!}, 
\end{align*}
for any $T>0$. 
Thus, $\{h^{n+1,n'+1}\}_{n'}$ is Cauchy in $L^\infty([0,T]\times\Omega\times\R^3)$, and its limit $h^{n+1}$ is continuous in $[0,\infty)\times\{\overline\Omega\times\R^3\setminus\gamma_0\}$. Then $h$ preserves the continuity in $[0,\infty)\times\{\overline\Omega\times\R^3\setminus\gamma_0\}$ from the uniform convergence. 

\smallskip 

The positivity can be proved in the same way as in \cite[pp. 802]{Guo2009} and the details are omitted for brevity.  This then completes the proof of Theorem \ref{Maincutoff}.	
\end{proof}

\medskip
\noindent {\bf Acknowledgements.} 
DQD was supported by Direct Grant from BIMSA and YMSC. 
RJD was partially supported by the NSFC/RGC Joint Research Scheme (N\_CUHK409/19) from RGC in Hong Kong and a Direct Grant from CUHK.

\end{document}